\documentclass[12pt,a4paper,reqno]{article}
\usepackage[T1]{fontenc}
\usepackage{authblk}
\pdfoutput=1
\usepackage{amsmath}
\usepackage{amsfonts,amsthm,amssymb}
\usepackage{mathabx}
\usepackage{bm}
\usepackage{amscd}
\usepackage[latin2]{inputenc}
\usepackage{t1enc}
\usepackage[mathscr]{eucal}
\usepackage{indentfirst}
\usepackage{graphicx}
\usepackage{graphics}
\usepackage{pict2e}
\usepackage{epic}
\numberwithin{equation}{section}
\usepackage[margin=2cm]{geometry}
\usepackage{epstopdf} 
\usepackage[colorlinks,linkcolor=blue]{hyperref}
\usepackage[capitalise,noabbrev]{cleveref}
\usepackage{todonotes}

\crefformat{equation}{(#2#1#3)}
\crefmultiformat{equation}{(#2#1#3)}{ and~(#2#1#3)}{, (#2#1#3)}{ and~(#2#1#3)}
\crefrangeformat{equation}{(#3#1#4) to~(#5#2#6)}

\usepackage[normalem]{ulem}

\allowdisplaybreaks

\theoremstyle{plain}
\newtheorem{Thm}{Theorem}[section]
\newtheorem*{Thm*}{Theorem}
\newtheorem{Lem}[Thm]{Lemma}

\newtheorem{Prop}[Thm]{Proposition}

\theoremstyle{definition}

\newtheorem{Rem}[Thm]{Remark}
\newtheorem{?}[Thm]{Problem}

\newcommand{\ovl}{\overline}
\newcommand{\p}{\partial}
\newcommand{\R}{\mathbb{R}}
\newcommand{\e}{\varepsilon}

\newcommand{\ovlv}{\ovl{v}}
\newcommand{\ovlu}{\ovl{u}}
\newcommand{\ovlvl}{\ovl{v}_l}
\newcommand{\ovlvr}{\ovl{v}_r}
\newcommand{\ovlul}{\ovl{u}_l}
\newcommand{\ovlur}{\ovl{u}_r}
\newcommand{\X}{\mathcal{X}}
\newcommand{\Y}{\mathcal{Y}}
\newcommand{\vt}{\tilde{v}}
\newcommand{\ut}{\tilde{u}}

\newcommand{\Torus}{\mathbb{T}}

\newcommand{\vn}{\mathbf{v}}
\newcommand{\un}{\mathbf{u}}
\newcommand{\nn}{\mathbf{n}}
\newcommand{\Wn}{\mathbf{W}}
\newcommand{\wn}{\mathbf{w}}
\newcommand{\vtn}{\tilde{\mathbf{v}}}
\newcommand{\utn}{\tilde{\mathbf{u}}}
\newcommand{\ntn}{\tilde{\mathbf{n}}}
\newcommand{\fn}{\mathbf{f}}
\newcommand{\Fn}{\mathbf{F}}
\newcommand{\Hn}{\mathbf{H}}
\newcommand{\Qn}{\mathbf{Q}}

\newcommand{\norm}[1]{\lVert#1\rVert}

\newcommand{\abs}[1]{\left\lvert#1\right\rvert}

\date{}

\begin{document}
	
	\begin{titlepage}
		\title{Stability of large-amplitude viscous shock under periodic perturbation for 1-d isentropic Navier-Stokes equations} 
		
		\author{Feimin Huang $ ^{1, 2} $}
		
		\author{Qian Yuan $ ^{1,} $\thanks{Corresponding author. \\
				Feimin Huang is partially supported by NSFC Grant No. 11371349 and 11688101. Qian Yuan is supported by the China Postdoctoral Science Foundation funded projects 2019M660831 and 2020TQ0345.}
		}
		
		\affil{\footnotesize $ ^1 $ Academy of Mathematics and Systems Science, Chinese Academy of Sciences, Beijing 100190, China. \\
			E-mails: fhuang@amt.ac.cn (F. Huang), qyuan@amss.ac.cn (Q. Yuan)}
		
		\affil{\footnotesize $ ^2 $ School of Mathematical Sciences, University of Chinese Academy of Sciences, Beijing 100049, China. }
		
	\end{titlepage}
	
	\maketitle
	
	\begin{abstract} 
		The stability of solutions under periodic perturbations for both inviscid and viscous conservation laws is an interesting and important problem.
		In this paper, a large-amplitude viscous shock under space-periodic perturbation for the isentropic Navier-Stokes equations is considered. It is shown that if the initial perturbation around the shock is suitably small and satisfies a zero-mass type condition \cref{zero-mass}, then the solution of the N-S equations tends to the viscous shock with a shift, which is partially determined by the periodic oscillations. 
		In other words, the viscous shock is nonlinearly stable even though the perturbation oscillates at the far fields. 
		The key point is to construct a suitable ansatz $ (\vt,\ut) $, which carries the same oscillations of the solution $ (v,u) $ at the far fields, so that the difference $ (v-\vt,u-\ut) $ belongs to the $ H^2(\R) $ space for all $ t\geq 0. $
	\end{abstract}

	\tableofcontents


\section{Introduction}

The one-dimensional (1-d) compressible isentropic Navier-Stokes (N-S) equations in the Lagrangian coordinates read
\begin{equation}\label{N-S}
\begin{cases}
\p_t v - \p_x u = 0, & \quad \\
\p_t u + \p_x p(v) = \p_x \big( \frac{\mu(v)}{v} \p_x u \big), & \quad
\end{cases} x\in \R, t>0,
\end{equation}
where $ v>0 $ is the specific volume, $ u \in\R $ is the velocity, the pressure $ p(v) $ is assumed to be smooth and satisfy
\begin{equation}\label{pressure}
 p'(v)<0 \ \text{and } p''(v)>0 \quad \text{for all } v>0,  
\end{equation}
and the viscosity $ \mu(v)>0 $ is a smooth function. It is noted that the polytropic gas (i.e. $ p(v) = v^{-\gamma}, \gamma>1 $) is included.
For later use, we define the function
\begin{equation}\label{sigma}
\sigma(v) := \int_1^{v} \frac{\mu(s)}{s} ds.
\end{equation}
In this paper, we are concerned about a Cauchy problem for \cref{N-S} with the initial data
\begin{equation}\label{ic}
(v,u)(x,0) = (v_0,u_0)(x), \quad x\in\R,
\end{equation}
satisfying
\begin{equation}\label{end-behavior}
(v_0,u_0)(x) \rightarrow \begin{cases}
(\ovlvl,\ovlul)+\left(\phi_{0l}, \psi_{0l} \right)(x) & \quad \text{as } x\rightarrow -\infty,\\
(\ovlvr,\ovlur)+(\phi_{0r}, \psi_{0r})(x) & \quad \text{as } x\rightarrow +\infty,
\end{cases}
\end{equation}
where $ \ovlvl>0, \ovlvr>0, \ovlul $ and $ \ovlur $ are constants, $ (\phi_{0l}, \psi_{0l}) \in \R^2 $ and $ (\phi_{0r}, \psi_{0r}) \in \R^2 $ are periodic functions with periods $ \pi_l>0 $ and $ \pi_r>0, $ respectively. 

We assume that the constant states $ (\ovlvl,\ovlul) $ and $ (\ovlvr,\ovlur) $ generate a single 2-shock (the case for 1-shock is similar), i.e. there hold the Rankine-Hugoniot condition,
\begin{equation}\label{R-H}
	\begin{cases}
		-s\left(\ovlvr-\ovlvl\right) - \left(\ovlur-\ovlul\right) = 0, & \\
		-s\left(\ovlur-\ovlul\right)+\left(p(\ovlvr)-p(\ovlvl)\right) = 0, &
	\end{cases}
\end{equation}
and the Lax's entropy condition,
\begin{equation}\label{Lax-entropy}
	\ovlvl<\ovlvr, \quad \ovlul > \ovlur,
\end{equation} 
where $ s $ is the shock speed, given by
\begin{equation}\label{shock-speed}
	s = \sqrt{-\frac{p(\ovlvr)-p(\ovlvl)}{\ovlvr-\ovlvl}}>0.
\end{equation}
Besides, we assume that the periodic perturbations in \cref{end-behavior} have zero averages,
\begin{equation}\label{zero-ave}
\int_{0}^{\pi_l} (\phi_{0l}, \psi_{0l})(x) dx = 0 \quad \text{ and } \quad \int_{0}^{\pi_r} (\phi_{0r}, \psi_{0r})(x) dx =0. 
\end{equation}
It is noted that if \cref{zero-ave} does not hold, the problem \cref{N-S,end-behavior} turns to be concerned with another Riemann solution, which is not the topic of this paper.

\vspace{0.1cm}

A viscous shock profile $ (v^S,u^S)(\xi) = (v^S,u^S)(x-st) $ is a traveling wave solution to \cref{N-S}, which tends to the constant state $ (\ovlvl,\ovlul) $ (resp. $ (\ovlvr,\ovlur) $) as $ \xi \to -\infty $ (resp. $ +\infty $), and solves the following problem,
\begin{equation}\label{ode-1}
	\begin{cases}
		-s\left(v^S\right)'(\xi)- \left(u^S\right)'(\xi) = 0, \qquad\qquad\qquad\qquad\qquad \xi \in \R,& \\
		-s\left(u^S\right)'(\xi) + \big(p(v^S)\big)'(\xi)=\big(\sigma'\left(v^S\right) \big(u^S\big)' \big)'(\xi), \qquad \xi \in \R,& \\
		\left(v^S,u^S\right)(\xi) \to \left(\ovlvl,\ovlul\right) \ \left(\text{resp. } (\ovlvr,\ovlur)\right) \qquad \text{as } \xi \to -\infty \ (\text{resp. } +\infty).
	\end{cases}
\end{equation}

\begin{Lem}[\cite{He2019a}, Theorem 1]\label{Lem-cite}
	Under the condition \cref{pressure}, there exists a unique smooth solution $ (v^S,u^S) $ to the system \cref{ode-1} up to a shift. Moreover, it holds that $ (v^S)'>0 $ and $ (u^S)' < 0. $
\end{Lem}

When the periodic perturbations vanish, the initial data $ (v_0,u_0) $
tends to constant states at the far fields. For this kind of  initial data, there have been various works about the global existence and asymptotic behavior of the solutions to the 1-d viscous conservation laws, including \cref{N-S}.
For the scalar equations, Il'in-Ole\v{\i}nik \cite{Oleinik1960} used the maximum principle to prove the $ L^\infty $-stability of constants, shocks and rarefaction waves in both inviscid and viscous cases. We also refer to Freist\"{u}hler-Serre \cite{Freistuhler1998} for the $ L^1 $-stability of scalar viscous shocks, which does not hold for systems as there appear diffusion waves propagating along other characteristic families; see \cite{Liu1985Mem}.
For the systems,  \cite{Matsumura1985,Goodman1986} proved the stability of weak viscous shocks for the Navier-Stokes equations and general parabolic systems, respectively, provided that the initial data satisfy the zero-mass condition. 
After that,  \cite{Liu1985Mem,Szepessy1993,Liu2015} removed the zero-mass condition by introducing diffusion waves to carry excessive masses and using point-wise estimates, which complete the nonlinear stability of weak viscous shocks for general systems of conservation laws under generic perturbations that decay at the far field. 
However, until now, the stability result of large-amplitude shock
	for general viscous conservations laws is still limited. In the works \cite{Zumbrun1998,MZ2004,HRZ2006} of K. Zumbrun, et al, it was shown that if viscous shocks are spectrally stable, then the nonlinear stability holds true. Then the following works \cite{HLyZ2009,HLaZ2010,BZ2016}
	verified the spectral stability of viscous shocks for the Navier-Stokes equations, with the aid of numeric analysis or high Mach numbers. Recently,
	with the aid of the effective velocity, \cite{He2019a} successfully used the elementary energy method to obtain the nonlinear stability of viscous shocks for the isentropic Navier-Stokes equations.   We also refer to \cite{Matsumura2010,Vasseur2016} for more interesting results on the stability of viscous shocks.

The stability of Riemann solutions under periodic perturbations for both inviscid and viscous conservation laws is an interesting and important problem. 
Lax \cite{Lax1957} and Glimm-Lax \cite{Glimm1970} were the first to study the periodic solutions to hyperbolic conservation laws. They showed that for the scalar equations and $ 2\times2 $ systems, the periodic solutions time-asymptotically decay to their constant averages. For the general Riemann initial data with periodic perturbations, recent papers \cite{Xin2019,XYY2019,YY2019} proved the asymptotic stability of shocks and rarefaction waves for the 1-d scalar conservation laws in both inviscid and viscous cases. And \cite{HY2020} used energy estimates to extend the stability result of the scalar rarefaction wave to the multi-dimensional case.

In this paper, we are concerned about the stability of a single shock with large amplitude under periodic perturbations for the N-S equations. It is shown that if the initial perturbation around the shock is suitably small and satisfies a zero-mass type condition \cref{zero-mass}, then the solution of the N-S equations tends to the viscous shock with a shift, which is partially determined by the periodic oscillations. The precise statement of the main result is given in \cref{Thm-shock} in the next section. 
A key point in the proof is to construct a suitable ansatz $ (\vt,\ut) $ by selecting two appropriate shift curves $ \X(t) $ and $ \Y(t). $ The ansatz carries the same oscillations of the solution $ (v,u) $ at the far fields, so that the difference $ (v-\vt,u-\ut) $ belongs to the $ H^2(\R) $ space and thus the anti-derivative method is available. The strategy for the a priori estimates is outlined as follows:  we first follow the idea of \cite{He2019a} to consider the equations of the volume and the effective velocity to obtain a basic energy estimate concerning the large-amplitude shock, and then we turn back to \cref{N-S} to achieve the desired estimate of higher-order derivatives. 

\vspace{0.3cm}

This paper is organized as follows. In \cref{Sec-ansatz}, the ansatz $ (\vt,\ut) $ is carefully constructed and then the main result, \cref{Thm-shock}, is stated. \cref{Thm-shock} is proved in \cref{Sec-thm} through energy estimates for the anti-derivatives of the perturbations, and some complicated and tedious calculations are moved to \cref{Sec-prop-shift} for easy reading.

\section{Ansatz and Main Result}\label{Sec-ansatz}

In the following of this paper, we denote $ \norm{\cdot} = \norm{\cdot}_{L^2(\R)} $ and $ \norm{\cdot}_k = \norm{\cdot}_{H^k(\R)} $ for $ k\geq 1. $ 

Let $\left( v_{l,r}, u_{l,r}\right)(x,t) $ be the unique periodic solution to \cref{N-S} with the periodic initial data 
\begin{equation*}
\left( v_{l,r}, u_{l,r}\right)(x,0) = \left(\ovl{v}_{l,r}, \ovl{u}_{l,r} \right) + \left(\phi_{0l,0r}, \psi_{0l,0r}\right)(x),
\end{equation*}
where $ \left(\phi_{0l,0r}, \psi_{0l,0r}\right) $ is given in \cref{end-behavior}.
We first give some properties of periodic solutions to \cref{N-S}.
\begin{Lem}\label{Lem-periodic}
	Assume that $ (v_0,u_0)(x)\in H^k(0,\pi) $ with $ k\geq 2 $ is periodic with period $ \pi>0 $ and average $ (\ovl{v},\ovl{u}). $ Then there exists $ \e_0>0 $ such that if 
	\begin{equation*}
		\e:=\norm{(v_0,u_0)-(\ovl{v},\ovl{u}) }_{H^k(0,\pi)} \leq \e_0,
	\end{equation*}
	the problem \cref{N-S} with initial data $ (v_0,u_0) $ admits a unique periodic solution $$ (v,u)(x,t) \in C(0,+\infty;H^k(0,\pi)), $$ which has same period and average as $ (v_0,u_0). $ Moreover, it holds that
	\begin{equation}\label{decay-per}
		\norm{(v,u)-(\ovl{v},\ovl{u})}_{H^k(0,\pi)}(t) \leq C \e e^{-\alpha t}, \quad t\geq 0,
	\end{equation}
	where the constants $ C>0 $ and $ \alpha>0 $ are independent of $ \e $ and $ t. $
\end{Lem}
\cref{Lem-periodic} can be proved by standard energy estimates with the Poincar\'{e} inequality. And the proof is left in the appendix. 

For a viscous shock profile $ \left(v^S,u^S\right)(x-st) $ solving \cref{ode-1}, let
\begin{equation}\label{g}
g(x) := \frac{v^S(x)-\ovlvl}{\ovlvr-\ovlvl}= \frac{u^S(x)-\ovlul}{\ovlur-\ovlul},
\end{equation}
where the equality can follow easily from \cref{R-H} and \cref{ode-1}.
It is straightforward to check from \cref{Lem-cite} that $ 0<g(x)<1 $ and $ g'(x) >0 $ for all $ x \in\R. $ 

For any function $ h(x) $ and curve $ \xi(t), $ we denote $ h_\xi $ to be the shifted function $$ h_\xi(x) := h(x-\xi(t)) \quad \text{with the derivatives } ~ h^{(k)}_\xi(x) := h^{(k)}(x-\xi(t)), \quad k\geq 1. $$ 

\textbf{Ansatz: } 
We look for our ansatz as in the form,
\begin{equation}\label{ansatz-shock}
\begin{aligned}
\vt(x,t) & := v_l(x,t) \left(1-g_{st+\X(t)}(x)\right) + v_r(x,t) g_{st+\X(t)}(x), \\
\ut(x,t) & := u_l(x,t) \left(1-g_{st+\Y(t)}(x)\right) + u_r(x,t) g_{st+\Y(t)}(x),
\end{aligned}
\end{equation}
where $ \X(t) $ and $ \Y(t) $ are two $ C^1 $ curves on $ [0,+\infty) $ to be determined. 
Note that $ (\vt,\ut) $ approaches the periodic solution $ (v_l,u_l) $ (resp. $ (v_r,u_r) $) as $ x\to -\infty $ (resp. $ +\infty $), thus we expect that the ansatz $ (\vt,\ut) $ carries same oscillations with the solution $ (v,u) $ at the far fields, i.e. $ (v-\vt,u-\ut)(x,t) \to 0 $ as $ \abs{x}\to +\infty $ for all $ t\geq 0. $ 

By plugging the ansatz $ (\vt,\ut) $ into \cref{N-S} with direct calculations, one can obtain that
\begin{equation}\label{eq-ansatz}
\begin{cases}
\p_t \vt - \p_x \ut = -\p_x F_1 - f_2, & \\
\p_t \ut + \p_x p(\vt) - \p_x\left(\sigma'(\vt) \p_x \ut \right) = -\p_x F_3 - f_4, &
\end{cases}
\end{equation}
where
\begin{equation}\label{source}
\begin{aligned}
F_1 =~ & (u_r-u_l)(g_{st+\Y}-g_{st+\X}), \\
f_2 =~ & (u_r-u_l)g'_{st+\X}+ (v_r-v_l) g'_{st+\X} (s+\X'),\\
F_3 =~ & - \left(p(\vt)- p(v_l)\right)(1-g_{st+\Y}) - \left(p(\vt)- p(v_r)\right) g_{st+\Y}  + \sigma'(\vt) (u_r-u_l) g'_{st+\Y}  \\
& +  \left(\sigma'(\vt)-\sigma'(v_l)\right)\p_x u_l (1-g_{st+\Y}) + \left( \sigma'(\vt) - \sigma'(v_r)\right)\p_x u_r g_{st+\Y}, \\
f_4 =~ & (u_r-u_l) g'_{st+\Y} (s+\Y') + \left(\sigma'(v_r) \p_x u_r - \sigma'(v_l)\p_x u_l \right) g'_{st+\Y} \\
& - \left(p(v_r) - p(v_l) \right) g'_{st+\Y}.
\end{aligned}
\end{equation}
It is noted that $ \vt-v_l = \left(v_r-v_l\right) g_{st+\X} $ and $ \vt-v_r = - \left(v_r-v_l\right) \left(1- g_{st+\X}\right), $ thus both $ F_1(x,t) $ and $ F_3(x,t) $ vanish as $ \lvert x\rvert \rightarrow +\infty $ for any $ t\geq 0. $ 

To make the anti-derivative method available, we aim to find appropriate shift curves $ \X(t) $ and $ \Y(t) $ so that 
\begin{equation}\label{zero-mass-2}
	\int_\R (v-\vt,u-\ut)(x,t) dx = 0 \quad \text{for all } t\geq 0.
\end{equation}
To achieve this, it follows from \cref{N-S,eq-ansatz} that
\begin{align*}
	\frac{d}{dt} \int_\R (v-\vt,u-\ut)(x,t) dx = \int_\R (f_2,f_4)(x,t)dx,
\end{align*}
which implies \cref{zero-mass-2}, if there hold
\begin{equation}\label{source-zero}
\int_\R  (f_2, f_4)(x,t) dx = 0, \quad t > 0,
\end{equation}
and 
\begin{equation}\label{zero-mass-ic}
	\int_\R \big( v_0(x)-\vt(x,0), u_0(x) - \ut(x,0) \big) dx = 0.
\end{equation}
From \cref{source-zero}, one has that
\begin{equation}\label{ode-shift}
\begin{aligned}
\X'(t) & =  - s- \dfrac{\int_\R (u_r-u_l)(x,t)g'_{st+\X}(x) dx}{\int_\R (v_r-v_l)(x,t) g'_{st+\X}(x) dx}, \\
\Y'(t) & =  -s+ \dfrac{\int_\R \left[ p(v_r)-p(v_l) - \sigma'(v_r) \p_x u_r + \sigma'(v_l) \p_x u_l \right](x,t) g'_{st+\Y}(x) dx}{\int_\R  \left(u_r-u_l\right)(x,t) g'_{st+\Y}(x) dx}.
\end{aligned}
\end{equation} 
From \cref{zero-mass-ic}, the initial data $ \X(0) = \X_0 $ and $ \Y(0) = \Y_0 $ can be uniquely determined.  Indeed, it follows from the first component of \cref{zero-mass-ic} that
\begin{align*}
	0 = & \int_\R \left(v_0- v^S_{\X_0} - \phi_{0l} (1-g_{\X_0})- \phi_{0r}g_{\X_0} \right)(x) dx \\
	= & \int_\R \left(v^S-v^S_{\X_0}\right)(x) dx + \int_{-\infty}^{0} (v_0-v^S-\phi_{0l})(x) dx + \int_{-\infty}^{0} \left(\phi_{0l}-\phi_{0r}\right)(x) g_{\X_0}(x) dx \\
	& + \int_{0}^{+\infty} \left(v_0-v^S-\phi_{0r}\right)(x) dx - \int_{0}^{+\infty} \left(\phi_{0l}-\phi_{0r}\right)(x) \left(1-g_{\X_0}(x)\right) dx,
\end{align*}
which yields that
\begin{equation}\label{eq-1}
	\mathcal{A}_1(\X_0) + \frac{1}{\ovlvr-\ovlvl} \Big\{  \int_{-\infty}^{0} \left(v_0-v^S-\phi_{0l}\right) dx + \int_{0}^{+\infty} \left(v_0-v^S-\phi_{0r}\right) dx \Big\} = 0,
\end{equation}
where
\begin{equation}\label{A1}
	\mathcal{A}_1(\X_0) := \X_0 + \frac{1}{\ovlvr-\ovlvl} \Big\{ \int_{-\infty}^{0} (\phi_{0l}-\phi_{0r}) g_{\X_0} dx - \int_{0}^{+\infty} (\phi_{0l}-\phi_{0r}) (1-g_{\X_0}) dx \Big\}.
\end{equation}
Similarly, one can get from the second component of \cref{zero-mass-ic} that $ \Y_0 $ satisfies
\begin{equation}\label{eq-2}
	\mathcal{A}_2(\Y_0) + \frac{1}{\ovlur-\ovlul} \Big\{ \int_{-\infty}^{0} \left(u_0-u^S-\psi_{0l}\right) dx + \int_{0}^{+\infty} \left(u_0-u^S-\psi_{0r}\right) dx\Big\} = 0, 
\end{equation}
where
\begin{equation}\label{A2}
	\mathcal{A}_2(\Y_0) := \Y_0 + \frac{1}{\ovlur-\ovlul} \Big\{ \int_{-\infty}^{0} (\psi_{0l}-\psi_{0r}) g_{\Y_0} dx - \int_{0}^{+\infty} (\psi_{0l}-\psi_{0r}) (1-g_{\Y_0}) dx \Big\}.
\end{equation}
Note that $ g'>0, $ then it follows from \cref{A1,A2} that
\begin{align*}
	\abs{\mathcal{A}_1'(\X_0) - 1} \leq \frac{\norm{\phi_{0l},\phi_{0r}}_{L^\infty(\R)}}{\abs{\ovlvr-\ovlvl}} \norm{g'}_{L^1(\R)} =  \frac{\norm{\phi_{0l},\phi_{0r}}_{L^\infty(\R)}}{\abs{\ovlvr-\ovlvl}},
\end{align*}	
and
\begin{align*}
	\abs{\mathcal{A}_2'(\Y_0) - 1} \leq   \frac{\norm{\psi_{0l},\psi_{0r}}_{L^\infty(\R)}}{\abs{\ovlur-\ovlul}}.
\end{align*}
Hence, if $ \norm{\phi_{0l}, \phi_{0r}, \psi_{0l}, \psi_{0r}}_{L^\infty(\R)} $ is small, one can get from implicit theorem that there exists a unique $ (\X_0, \Y_0) $ such that \cref{eq-1,eq-2} hold true.
Once the initial data $ (\X_0, \Y_0) $ is determined, one can solve the ODE system \cref{ode-shift}.

\begin{Lem}\label{Lem-shift}
	Assume that the periodic perturbations $ \phi_{0l},\phi_{0r}, \psi_{0l} $ and $ \psi_{0r} $ satisfy \cref{zero-ave}. Then there exists $ \e_0>0 $ such that if 
	\begin{equation*}
	\e := \norm{\phi_{0l}, \psi_{0l}}_{H^2(0,\pi_l)} + \norm{\phi_{0r}, \psi_{0r}}_{H^2(0,\pi_r)} \leq \e_0,
	\end{equation*}
	there exists a unique solution $ (\X, \Y)(t) \in C^1[0,+\infty) $ to \cref{ode-shift} with the initial data $  (\X_0, \Y_0). $ Moreover, the solution satisfies that
	\begin{equation*}
	\begin{aligned}
	& \abs{\X'(t)} + \abs{\X(t)-\X_\infty} \leq C\e e^{-\alpha t}, \\
	& \abs{\Y'(t)} + \abs{\Y(t)-\Y_\infty} \leq C \e e^{-\alpha t}, 
	\end{aligned} \qquad t\geq 0,
	\end{equation*}	
	where the constants $ C>0 $ and $ \alpha>0 $ are independent of $ \e $ and $ t, $ and the shifts $ \X_\infty $ and $ \Y_\infty $ are given by $ \X_\infty = \mathcal{A}_1(\X_0) + \mathcal{C}_1 $ and $ \Y_\infty = \mathcal{A}_2(\Y_0) + \mathcal{C}_2, $ with 
	\begin{equation}\label{X-inf}
	\begin{aligned}
	\mathcal{C}_1 & :=  \frac{1}{\ovlvr-\ovlvl} \Big\{ \frac{1}{\pi_l} \int_{0}^{\pi_l} \int_{0}^{x} \phi_{0l}(y)dy dx - \frac{1}{\pi_r} \int_{0}^{\pi_r} \int_{0}^{x} \phi_{0r}(y)dy dx \Big\},
	\end{aligned}
	\end{equation}
	and
	\begin{equation}\label{Y-inf}
	\begin{aligned} 
	\mathcal{C}_2 & := \frac{1}{\ovlur-\ovlul} \Big\{ \frac{1}{\pi_l} \int_{0}^{\pi_l} \int_{0}^{x} \psi_{0l}(y)dydx - \int_{0}^{+\infty} \frac{1}{\pi_l} \int_{0}^{\pi_l} \left[p(v_l(x,t))-p(\ovlvl)\right] dx dt \\ 
	& \quad -\frac{1}{\pi_r} \int_{0}^{\pi_r} \int_{0}^{x} \psi_{0r}(y)dydx  + \int_{0}^{+\infty} \frac{1}{\pi_r} \int_{0}^{\pi_r} \left[p(v_r(x,t))-p(\ovlvr)\right] dx dt \\
	& \quad +\sigma(\ovlvl) - \sigma(\ovlvr) - \frac{1}{\pi_l} \int_{0}^{\pi_l} \sigma\left(\ovlvl+\phi_{0l}(x) \right) dx + \frac{1}{\pi_r} \int_{0}^{\pi_r} \sigma\left(\ovlvr+\phi_{0r}(x) \right) dx \Big\}.
	\end{aligned}
	\end{equation}
\end{Lem}
The proof of \cref{Lem-shift} is placed in \cref{Sec-prop-shift}.
It is noted that all the integrals in \cref{X-inf,Y-inf} are bounded due to \cref{Lem-periodic}, and the constants $ \mathcal{C}_1 $ and $ \mathcal{C}_2 $ are independent of $ \X_0 $ and $ \Y_0. $

Note that as $ t\rightarrow +\infty, $ the ansatz $ (\vt,\ut) $ in \cref{ansatz-shock} tends to $ \left( v^S(x-st-\X_\infty),u^S(x-st-\Y_\infty) \right). $ This pair is a traveling wave solution to \cref{N-S} if and only if 
\begin{equation}\label{same-limit}
	\X_\infty = \Y_\infty,
\end{equation}
which unfortunately does not hold for generic initial perturbations. In order to ensure that \cref{same-limit} holds, we assume additionally that the initial data $ (v_0,u_0) $ satisfies \cref{same-limit}. In fact, by plugging \cref{eq-1} and \cref{eq-2} into \cref{X-inf} and \cref{Y-inf}, respectively, one can get that \cref{same-limit} is equivalent to 
\begin{equation}\label{zero-mass}
	\begin{aligned}
		& s ~\Big\{ \int_{-\infty}^{0} \left(v_0-v^S-\phi_{0l}\right)(x) dx + \int_{0}^{+\infty} \left(v_0-v^S-\phi_{0r}\right)(x) dx \\
		& \qquad - \frac{1}{\pi_l} \int_{0}^{\pi_l} \int_{0}^{x} \phi_{0l}(y) dydx + \frac{1}{\pi_r} \int_{0}^{\pi_r} \int_{0}^{x} \phi_{0r}(y) dydx \Big\} \\
		& \quad = -\int_{-\infty}^{0} \left(u_0-u^S-\psi_{0l}\right)(x) dx - \int_{0}^{+\infty} \left(u_0-u^S-\psi_{0r}\right)(x) dx  \\
		& \qquad + \frac{1}{\pi_l} \int_{0}^{\pi_l} \int_{0}^{x} \psi_{0l}(y) dydx - \int_{0}^{+\infty} \frac{1}{\pi_l} \int_{0}^{\pi_l} \left[ p(v_l(x,t)) - p(\ovlvl) \right] dx dt  \\
		& \qquad - \frac{1}{\pi_r} \int_{0}^{\pi_r} \int_{0}^{x} \psi_{0r}(y) dydx + \int_{0}^{+\infty} \frac{1}{\pi_r} \int_{0}^{\pi_r} \left[ p(v_r(x,t)) - p(\ovlvr) \right] dx dt \\
		& \qquad + \sigma(\ovlvl) - \sigma(\ovlvr) - \frac{1}{\pi_l} \int_{0}^{\pi_l} \sigma(\ovlvl+\phi_{0l}(x)) dx + \frac{1}{\pi_r} \int_{0}^{\pi_r} \sigma(\ovlvr+\phi_{0r}(x)) dx.
	\end{aligned}
\end{equation}
We call \cref{zero-mass} as a zero-mass type condition.

\begin{Rem} 
	If all the periodic perturbations $ \phi_{0l}, \phi_{0r}, \psi_{0l} $ and $ \psi_{0r} $ vanish, \cref{zero-mass} turns to
	\begin{equation*}
	s \int_\R  \left( v_0-v^S\right) dx + \int_\R  \left(u_0-u^S\right) dx =0.
	\end{equation*}
	This condition is equivalent to that there exists a (unique) constant $ \bar{x}\in\R $ such that
	\begin{equation*}
	\int_\R \left( v_0(x)-v^S(x-\bar{x})\right) dx = 0, \quad \int_\R \left( u_0(x)-u^S(x-\bar{x}) \right) dx = 0,
	\end{equation*}
	which is indeed the zero-mass condition given in \cite{Matsumura1985,Goodman1986}. This is why we denote \cref{zero-mass} as a zero-mass type condition.
\end{Rem}

\begin{Rem}
	In the special case that $ v_0 = v^S + \phi_0 $ and $ u_0 = u^S + \psi_0, $ where $ (\phi_0,\psi_0) $ is periodic with period $ \pi, $ \cref{zero-mass} is reduced to
	\begin{align*}
	0 & = - \int_{0}^{+\infty} \frac{1}{\pi} \int_{0}^{\pi} \left( p(v_l) - p(\ovlvl) \right) dx dt + \int_{0}^{+\infty} \frac{1}{\pi} \int_{0}^{\pi} \left( p(v_r) - p(\ovlvr) \right) dx dt \\
	& \quad + \sigma(\ovlvl) - \sigma(\ovlvr) - \frac{1}{\pi} \int_{0}^{\pi} \sigma\left(\ovlvl+\phi_0 \right) dx + \frac{1}{\pi} \int_{0}^{\pi} \sigma\left(\ovlvr+\phi_0 \right) dx.
	\end{align*}
\end{Rem}

Once the ansatz is well constructed, we define the anti-derivative of the difference between the initial data \cref{ic} and ansatz,
\begin{equation}\label{Phsi-0}
\Phi_0(x) := \int_{-\infty}^{x} \left( v_0(y)-\vt(y,0)\right) dy, \quad  \Psi_0(x) := \int_{-\infty}^{x} \left(u_0(y)-\ut(y,0) \right) dy,
\end{equation}
i.e.
\begin{align*}
\Phi_0(x) & = \int_{-\infty}^x \big[  (v_0-\ovlvl-\phi_{0l}) + (\ovlvr-\ovlvl+\phi_{0r}-\phi_{0l}) g_{\X_0}\big](y) dy, \\
\Psi_0(x) & = \int_{-\infty}^x \big[  (u_0-\ovlul-\psi_{0l}) + (\ovlur-\ovlul+\psi_{0r}-\psi_{0l}) g_{\Y_0}\big](y) dy.
\end{align*}
Set 
\begin{equation}\label{E0-shock}
E_0 := \norm{\phi_{0l}, \psi_{0l}}_{H^3(0,\pi_l)} + \norm{\phi_{0r}, \psi_{0r}}_{H^3(0,\pi_r)} + \norm{\Phi_0}_{H^2(\R)} + \norm{ \Psi_0}_{H^2(\R)}.
\end{equation}
Now we are ready to give the main result of this paper.

\begin{Thm}\label{Thm-shock}
	Assume that the periodic perturbations in \cref{end-behavior} satisfy \cref{zero-ave,zero-mass}. Then there exists $ \e_0>0 $ such that, if $ E_0 \leq \e_0, $ there admits a unique global solution of \cref{N-S,ic}, satisfying
	\begin{align*}
	& v-\vt \in C\left([0,+\infty);H^1(\R)\right) \cap L^2\left((0,+\infty);H^1(\R)\right), \\
	& u-\ut \in C\left([0,+\infty);H^1(\R)\right) \cap L^2\left((0,+\infty);H^2(\R)\right),
	\end{align*}
	and
	\begin{equation}\label{ineq-Thm}
	\norm{ (v,u)(\cdot,t) - (v^S,u^S)(\cdot-st-\X_\infty) }_{L^\infty(\R)} \rightarrow 0 \quad \text{ as } t\rightarrow +\infty,
	\end{equation}
	with the constant shift $ \X_\infty \left(= \Y_\infty\right) $ given by 
	\begin{equation*}
	\begin{aligned}
	\X_\infty = &~ -\frac{1}{\ovlvr-\ovlvl} \Big\{ \int_{-\infty}^{0} \left(v_0-v^S-\phi_{0l}\right)(x) dx + \int_{0}^{+\infty} \left(v_0-v^S-\phi_{0r}\right)(x)dx \Big\} + \mathcal{C}_1,
	\end{aligned}
	\end{equation*}
where $ \mathcal{C}_1 $ is given in \cref{X-inf}.
\end{Thm}

\begin{Rem}
	In \cref{Thm-shock}, the strength of the shock wave can be arbitrarily large.
\end{Rem}

\begin{Rem}
	It is an interesting problem to remove the zero-mass type condition \cref{zero-mass}.
\end{Rem}

\section{Proof of \cref{Thm-shock}}\label{Sec-thm}

For convenience, we change the coordinates $ (x,t) \rightarrow (\xi = x-st,t). $  In the following, we write the bold letters to denote the functions under the new coordinates $ (\xi,t), $ e.g.
\begin{equation*}
\vn(\xi,t) = v(\xi+st,t), \quad \vtn(\xi,t) = \vt(\xi+st,t).
\end{equation*}
Define the perturbation terms 
\begin{equation*}
\phi(\xi,t) = \left(\vn-\vtn\right)(\xi,t), \quad  \psi(\xi,t) = \left(\un-\utn\right)(\xi,t).
\end{equation*}
It follows from \cref{N-S,eq-ansatz} that
\begin{equation}\label{equ-phi-psi}
	\begin{cases}
		\p_t \phi - s\p_\xi \phi - \p_\xi \psi = \p_\xi \Hn_1, \\
		\p_t \psi - s\p_\xi \psi + \p_\xi (p(\vn)-p(\vtn)) = \p_\xi \big(\sigma'(\vn) \p_\xi \un - \sigma'(\vtn) \p_\xi \utn \big) + \p_\xi \Hn_2,
	\end{cases}
\end{equation}
where 
\begin{equation}\label{H-1-2}
	\Hn_1 := \Fn_1 + \Fn_2, \quad  \Hn_2 := \Fn_3 + \Fn_4,	
\end{equation}
with $ \Fn_1(\xi,t) = F_1(\xi+st,t) $ and $ \Fn_3(\xi,t) = F_3(\xi+st,t) $ given by \cref{source}, and 
$$ \Fn_2(\xi,t) := \int_{-\infty}^{\xi} \fn_2(y,t) dy = \int_{-\infty}^{\xi} f_2(y+st,t) dy = - \int_\xi^{+\infty} f_2(y+st,t) dy, $$ 
$$ \Fn_4(\xi,t) := \int_{-\infty}^{\xi} \fn_4(y,t) dy = \int_{-\infty}^{\xi} f_4(y+st,t) dy = - \int_\xi^{+\infty} f_4(y+st,t) dy. $$
Define the anti-derivative variables
\begin{equation}\label{Phi-Psi}
\Phi(\xi,t) := \int_{-\infty}^{\xi} \phi(y,t) dy, \quad  \Psi(\xi,t) := \int_{-\infty}^{\xi} \psi(y,t) dy.
\end{equation} 
From \cref{equ-phi-psi}, it holds that
\begin{equation}\label{equ-Phi-Psi}
\begin{cases}
\p_t \Phi - s\p_\xi \Phi - \p_\xi \Psi = \Hn_1, \\
\p_t \Psi - s\p_\xi \Psi + p'(\vtn) \p_\xi \Phi = - p(\vn|\vtn) + \sigma'(\vn) \p_\xi^2 \Psi + \big(\sigma'(\vn) - \sigma'(\vtn)\big) \p_\xi \utn + \Hn_2,
\end{cases}
\end{equation}
with initial data
\begin{equation}\label{ic-anti}
	(\Phi,\Psi)(\xi,0) = (\Phi_0, \Psi_0)(\xi) \in H^2(\R)\times H^2(\R),
\end{equation}
where
\begin{equation}\label{nonlinear-p}
	p(\vn|\vtn) := p(\vn)-p(\vtn)-p'(\vtn) \p_\xi \Phi \geq 0.
\end{equation}
For any $ T>0, $ let 
\begin{equation*}
	\begin{aligned}
		\mathcal{B}(0,T) = \Big\{ & (\Phi,\Psi) \in C\left(0,T;H^2(\R)\times H^2(\R)\right): \\
		& \qquad \p_\xi \Phi \in L^2 \left(0,T;H^1(\R)\right), \p_\xi \Psi \in L^2\left(0,T;H^2(\R)\right)\Big\}.
	\end{aligned}
\end{equation*}
We aim to seek the solution $ (\Phi,\Psi) $ of \cref{equ-Phi-Psi,ic-anti} in the space $ \mathcal{B}(0,+\infty) $ provided that the initial perturbations are small.

\begin{Thm}\label{Thm-pert}
Under the assumptions of \cref{Thm-shock}, there exists $ \e_0>0, $ such that if $ \e= E_0\leq \e_0, $ then  there is a unique solution $ (\Phi,\Psi) \in \mathcal{B}(0,+\infty) $  of \cref{equ-Phi-Psi,ic-anti}, satisfying
\begin{equation}\label{0-infty}
	\sup_{t >0} \norm{ \Phi,\Psi }_2^2(t) + \int_{0}^{+\infty} \left(\norm{\p_\xi \Phi}_1^2 + \norm{\p_\xi \Psi}_2^2 \right)(t) dt \leq C \e.
\end{equation}
Moreover, it holds that
\begin{equation}\label{0limit}
	\norm{\phi,\psi}_{L^\infty(\R)}(t) \to 0 \quad \text{ as } t\to +\infty.
\end{equation}
\end{Thm}
Once \cref{Thm-pert} is proved, to complete the proof of \cref{Thm-shock}, it suffices to show \cref{ineq-Thm}. In fact, since
\begin{equation*}
	\begin{aligned}
		\sup_{x\in\R} \abs{(v-\vt)(x,t)} & =
		\sup_{\xi\in\R} \abs{(\vn-\vtn)(\xi,t)} = \norm{\phi}_{L^\infty(\R)} \to 0 \quad \text{as } t\rightarrow +\infty,
	\end{aligned}
\end{equation*}
and
\begin{equation}\label{ineq-v}
	\begin{aligned}
		& \sup_{x\in\R} \abs{\vt(x,t)-v^S\left(x-st-\X_\infty\right)} \\
		& \quad \leq  \sup_{x\in\R} \abs{ \left( v^S(x-st-\X(t)) - v^S(x-st-\X_\infty) \right) } \\
		& \qquad + \sup_{x\in\R} \abs{\left(v_l-\ovlvl\right)(1-g_{\X}) + \left(v_r-\ovlvr\right)g_{\X} }  \\
		& \quad \leq C \abs{\X(t)-\X_\infty} + \sup_{x\in\R} \abs{v_l-\ovlvl} + \sup_{x\in\R} \abs{v_r-\ovlvr} \\
		& \quad \leq Ce^{-\alpha t},
	\end{aligned}
\end{equation}
thus \cref{ineq-Thm} holds for $ v. $ And the estimate for $ u $ in \cref{ineq-Thm} can be proved similarly.

It remains to prove \cref{Thm-pert}. 
We first show the following decay properties of the error terms in \cref{equ-Phi-Psi}.

\begin{Lem}\label{Lem-H}
	Under the assumptions of \cref{Thm-shock}, it holds that 
	\begin{equation*}
		\norm{\Hn_1}_2(t), \norm{\Hn_2}_1(t) \leq C \e e^{-\alpha t}, \quad t\geq 0,
	\end{equation*}
	where $ C>0 $ is independent of $ \e $ and $ t, $ and $ \alpha>0 $ is the constant given in \cref{Lem-shift}.
\end{Lem}
The proof of \cref{Lem-H} is placed in \cref{Sec-H}.
Now we aim to achieve the following a priori estimates.

\begin{Prop}[A priori estimates]\label{Prop-apriori}
	For any $ T>0, $ assume that $ (\Phi,\Psi) \in \mathcal{B}(0,T) $ is a solution of \cref{equ-Phi-Psi,ic-anti}. Then there exist $ \e_0>0 $ and $ \delta_0>0, $ independent of $ T, $ such that if
	\begin{equation}\label{assum-apriori}
	\e = E_0 <\e_0 \quad \text{and} \quad
	\delta = \sup_{t \in (0,T)} \norm{(\Phi,\Psi)}_2(t) <\delta_0,
	\end{equation}
	then
	\begin{equation}\label{apriori}
	\sup_{t \in (0,T)} \norm{(\Phi,\Psi)}_2^2(t) + \int_{0}^{T} \left(\norm{\p_\xi \Phi}_1^2 + \norm{\p_\xi \Psi}_2^2 \right)(t) d\tau \leq C_0 \big(\norm{\Phi_0,\Psi_0}_2^2 + \e\big),
	\end{equation}
	where $ C_0>0 $ is independent of $ \e,\delta $ and $ T. $
\end{Prop}

\subsection{Proof of \cref{Prop-apriori}}
In order to deal with the large-amplitude shock, we first follow the idea of \cite{He2019a} to consider the so-called effective velocity, $ \nn:= \un -\p_\xi \sigma(\vn), $ instead of the velocity $ \un, $ to obtain the energy estimate of lower-order terms; see Lemmas \ref{Lem-esti-0} and \ref{Lem-esti-1}. Then we turn back to \cref{equ-Phi-Psi} to obtain the higher-order estimate; see \cref{Lem-Psi}.

Define $ \ntn := \utn -\p_\xi \sigma(\vtn), $
and denote the perturbation term as $ \wn = \nn-\ntn, $ with the anti-derivative variable
\begin{equation}\label{W-def}
	\Wn(\xi,t) := \int_{-\infty}^{\xi} \wn(y,t) dy = \Psi(\xi,t) - \big(\sigma(\vn)-\sigma(\vtn)\big)(\xi,t).
\end{equation}
With this new variable, it follows from \cref{N-S,eq-ansatz} that
\begin{equation}\label{equ-anti}
	\begin{cases}
		\p_t \Phi - s \p_\xi \Phi - \p_\xi \Wn - \sigma'(\vtn) \p_\xi^2 \Phi = \Qn + \Hn_1, &\\
		\p_t \Wn - s\p_\xi \Wn + p'(\vtn) \p_\xi \Phi = - p(\vn|\vtn) +  \Hn_2-\sigma'(\vtn) \p_\xi \Hn_1,
	\end{cases}
\end{equation}
where $ \Qn $ is a nonlinear term given by
\begin{align}
	\Qn & := \left(\sigma'(\vn)-\sigma'(\vtn) \right) \left( \p_\xi \vtn + \p_\xi^2 \Phi \right). \label{Q}
\end{align}
The initial data of $ \Wn $ is 
\begin{equation*}
	\Wn_0(\xi) := \Wn(\xi,0) = \Psi_0(\xi) - \big(\sigma(\vn)-\sigma(\vtn)\big)(\xi,0) \in H^1(\R),
\end{equation*}
which satisfies that $ \norm{\Wn_0}_{H^1(\R)} \leq \norm{ \Psi_0}_{H^1(\R)} + C \norm{\Phi_0}_{H^2(\R)} \leq C\e. $

\vspace{0.2cm}

Under the assumption that $ \e_0 $ and $ \delta_0 $ are both small enough, it follows from \cref{assum-apriori} and the Sobolev inequality that 
\begin{equation}\label{bdd-apriori}
	\begin{aligned}
	& \sup_{t \in (0,T)} \norm{\Phi,\Psi}_{W^{1,\infty}(\R)} \leq C \sup_{t \in (0,T)} \norm{\Phi,\Psi}_2 \leq C \delta_0, \\
	& \sup_{t \in (0,T)} \norm{\Wn}_{L^\infty(\R)} \leq C \sup_{t \in (0,T)} \norm{\Wn}_1 \leq C\delta_0,
	\end{aligned}	
\end{equation}
and $ \ovlvl/2 \leq \vtn(\xi,t) \leq 2 \ovlvr, \quad \ovlvl/4 \leq \vn(\xi,t) \leq 4 \ovlvr $ for $ \xi\in \R $ and $ t\in [0,T]. $

\vspace{0.3cm}

\begin{Lem}\label{Lem-esti-0}
	Under the assumptions of \cref{Prop-apriori}, there exist $ \e_0>0 $ and $ \delta_0>0 $ such that if $ \e<\e_0 $ and $ \delta<\delta_0, $ then
	\begin{equation}\label{esti-0}
	\begin{aligned}
	& \sup_{t \in (0,T)} \norm{(\Phi,\Wn)}^2(t) + \int_{0}^{T} \Big( \norm{\p_\xi \Phi}^2 + \norm{\sqrt{-\left(u^S_\Y \right)'} ~ \Wn}^2 \Big)(t) dt \\
	& \qquad \leq C \Big( \norm{\Phi_0, \Wn_0}^2 + \e + \delta \int_{0}^{T} \norm{\p_\xi^2 \Phi}^2(t) dt \Big),
	\end{aligned}  
	\end{equation}
	where $ C>0 $ is independent of $ \e,\delta $ and $ T. $
\end{Lem}
\begin{proof}
	Multiplying $ \Phi $ on \cref{equ-anti}$_1 $ and $ - \frac{ \Wn}{p'(\vtn)} $ on \cref{equ-anti}$_2 $ yields that
	\begin{equation}\label{equality-1}
	\begin{aligned}
	& \p_t \Big( \frac{\Phi^2}{2} + \frac{ \Wn^2}{2\abs{p'(\vtn)}}  \Big) + \sigma'(\vtn) \left(\p_\xi \Phi \right)^2 - \frac{p''(\vtn)}{2 \left(p'(\vtn)\right)^2}\left(u^S_\Y\right)'  \Wn^2  \\
	& \quad =  \underbrace{ \frac{p''(\vtn)}{2(p'(\vtn))^2} \big( \p_\xi \utn - \left(u^S_\Y\right)' - \p_\xi \Hn_1 \big) }_{I_{1,1}} \Wn^2 + \underbrace{\Hn_1 \Phi - \frac{ \Wn}{p'(\vtn)} \big(\Hn_2-\sigma'(\vtn) \p_\xi \Hn_1\big)}_{I_{1,2}} \\
	& \qquad  + \underbrace{ \big( \Qn - \sigma''(\vtn)\p_\xi\vtn \p_\xi \Phi \big) \Phi + \frac{ p(\vn|\vtn)}{p'(\vtn)} \Wn}_{I_{1,3}} + \p_\xi \{\cdots\},
	\end{aligned}
	\end{equation}
	where $ \{\cdots\} = \frac{s}{2} \Phi^2 + \Phi \Wn + \sigma'(\vtn) \Phi \p_\xi\Phi - \frac{s}{2p'(\vtn)}  \Wn^2. $
	Since $ \utn(\xi,t) = \un_l(\xi,t) (1-g_{\Y}(\xi)) + \un_r(\xi,t) g_\Y(\xi) , $ it follows from Lemmas \ref{Lem-periodic} and \ref{Lem-shift} that
	\begin{align}
	\norm{\p_\xi \utn-\left(u^S_\Y\right)'}_{L^\infty(\R)} & \leq C \left( \norm{\p_\xi \un_l}_{L^\infty(\R)} + \norm{\p_\xi \un_r}_{L^\infty(\R)} + \norm{\un_l-\ovlul}_{L^\infty(\R)}  + \norm{\un_r-\ovlur}_{L^\infty(\R)} \right) \notag \\
	& \leq C\e e^{-\alpha t}. \label{term-1}
	\end{align}
	This, together with \cref{Lem-H}, yields that 
	\begin{equation}\label{I-1-1}
		\norm{I_{1,1}}_{L^\infty(\R)} \leq C\e e^{-\alpha t}. 
	\end{equation}
	It also follows from \cref{Lem-H} that
	\begin{equation}\label{I-1-2}
	\begin{aligned}
	\int_0^T \int_\R \abs{I_{1,2}} d\xi dt & \leq C \int_0^T \left[ \norm{\Hn_1} \norm{\Phi} + \big(\norm{\Hn_2}+\norm{\p_\xi \Hn_1}\big) \norm{ \Wn} \right] dt \\
	& \leq C \e \int_0^T e^{-\alpha t} \big(\norm{\Phi} + \norm{ \Wn} \big) dt \\
	& \leq C \e \sup_{t \in (0,T)} \big(\norm{\Phi}^2 + \norm{ \Wn}^2 \big) + C\e.
	\end{aligned}		
	\end{equation}
	By \cref{nonlinear-p,Q}, one has that $ 0\leq p(\vn,\vtn) \leq C\abs{\p_\xi \Phi}^2 $ and 
	$$ \abs{\Qn - \sigma''(\vtn)\p_\xi\vtn \p_\xi \Phi} \leq C \left( \abs{\p_\xi\Phi}\abs{\p_\xi^2 \Phi} + \abs{\p_\xi\Phi}^2 \right). $$ 
	Then by \cref{bdd-apriori}, the nonlinear term $ I_{1,3} $ satisfies that
	\begin{align*}
		\int_0^T \int_\R \abs{I_{1,3}} d\xi dt & \leq C \int_0^T \int_\R \left[ \norm{\Phi}_{L^\infty(\R)} \left( \abs{\p_\xi\Phi}\abs{\p_\xi^2 \Phi} + \abs{\p_\xi\Phi}^2 \right) + \norm{ \Wn}_{L^\infty(\R)} \abs{\p_\xi \Phi}^2 \right] d\xi dt \\
		& \leq C \delta \int_{0}^{T} \norm{\p_\xi \Phi}^2 dt + C \delta \int_{0}^{T} \norm{\p_\xi^2 \Phi}^2 dt.
	\end{align*}
	Collecting the estimates from $ I_{1,1} $ to  $ I_{1,3}, $ and using the fact that $ (u^S)'<0, $ one can integrate \cref{equality-1} over $ \R\times(0,T) $ to get that
	\begin{equation*}
	\begin{aligned}
	& \sup_{t \in (0,T)} \norm{\Phi, \Wn}^2 + \int_0^T \Big(\norm{\p_\xi \Phi}^2 + \norm{\sqrt{-(u^S_\Y)'} ~ \Wn}^2 \Big) dt \\
	& \quad \leq C\norm{\Phi_0, \Wn_0}^2 + C\e \sup_{t \in (0,T)} \norm{\Phi, \Wn}^2 + C\e + C \delta \int_{0}^{T} \norm{\p_\xi \Phi}^2 dt + C \delta \int_{0}^{T} \norm{\p_\xi^2 \Phi}^2 dt. 
	\end{aligned}
	\end{equation*}
	Thus, if $ \e_0>0 $ and $ \delta_0>0 $ are small enough, one can obtain \cref{esti-0}.
\end{proof}	

\vspace{0.05cm}
	
\begin{Lem}\label{Lem-esti-1}
	Under the assumptions of \cref{Prop-apriori}, there exist $ \e_0>0 $ and $ \delta_0>0 $ such that if $ \e<\e_0 $ and $ \delta<\delta_0, $ then
	\begin{equation}\label{esti-1}
	\begin{aligned}
	& \sup_{t \in (0,T)} \norm{\p_\xi (\Phi,\Wn)}^2(t) + \int_{0}^{T} \Big(\norm{\p_\xi^2 \Phi}^2(t) + \norm{\sqrt{-(u^S_\Y)'} \p_\xi \Wn}^2 \Big)(t) dt \\
	&\quad \leq C\big(\norm{\Phi_0, \Wn_0}_1^2 + \e\big),
	\end{aligned}  
	\end{equation}
	where $ C>0 $ is independent of $ \e,\delta $ and $ T. $
\end{Lem}
\begin{proof}
	Multiplying $ \p_\xi^2 \Phi $ on both sides of \cref{equ-anti}$ _1, $ and  differentiating \cref{equ-anti}$ _2 $ with respect to $ \xi $ and multiplying $ - \frac{\p_\xi  \Wn}{p'(\vtn)} $ on the resulting equation, one has that
	\begin{equation}\label{equ-2}
	\begin{aligned}
	& \p_t \Big(\frac{\abs{\p_\xi\Phi}^2}{2}  + \frac{ \abs{\p_\xi \Wn}^2}{2\abs{p'(\vtn)}} \Big) + \sigma'(\vtn) \abs{\p_\xi^2 \Phi}^2 - \frac{p''(\vtn)}{2\left(p'(\vtn)\right)^2} \left(u^S_\Y\right)' \abs{\p_\xi \Wn}^2  \\
	& \quad = \p_\xi \{\cdots\} + I_{1,1} \abs{\p_\xi \Wn}^2 \\
	& \qquad + \underbrace{\frac{p''(\vtn)}{p'(\vtn)} \big(\p_\xi\vtn - \left(v^S_\X\right)' \big)}_{I_{2,2}} \p_\xi\Phi \p_\xi\Wn + \underbrace{\frac{p''(\vtn)}{p'(\vtn)} \left(v^S_\X\right)' \p_\xi\Phi \p_\xi\Wn}_{I_{2,3}} \\
	& \qquad \underbrace{ - \Qn \p_\xi^2 \Phi + \p_\xi p(\vn|\vtn) \frac{\p_\xi  \Wn}{p'(\vtn)}}_{I_{2,4}} \underbrace{- \Hn_1 \p_\xi^2 \Phi - \p_\xi \big(\Hn_2-\sigma'(\vtn) \p_\xi \Hn_1\big) \frac{\p_\xi  \Wn}{p'(\vtn)}}_{I_{2,5}},
	\end{aligned}
	\end{equation}
	where $ \{\cdots\} = \p_t \Phi \p_\xi \Phi - \frac{s}{2} (\p_\xi \Phi)^2 - \frac{s}{2p'(\vtn)} (\p_\xi \Wn)^2, $ and $ I_{1,1} $ is the term given in \cref{equality-1}, satisfying \cref{I-1-1}.
	Similar to \cref{term-1}, one can prove that $ \norm{I_{2,2}}_{L^\infty(\R)} \leq C\e e^{-\alpha t}. $ It follows from \cref{ode-1} and \cref{Lem-shift} with $ \X_\infty = \Y_\infty $ that
	\begin{equation}\label{term-2}
		\norm{s(v^S_\X)' - (u^S_\Y)'}_{L^\infty(\R)} = \norm{(u^S_\X)' - (u^S_\Y)'}_{L^\infty(\R)} \leq C \abs{\X(t)-\Y(t)} \leq C \e e^{-\alpha t}.
	\end{equation}
	Thus, $ I_{2,3} $ satisfies that
	\begin{align*}
		\int_0^T \int_\R \abs{I_{2,3}} d\xi dt & \leq C \int_{0}^{T} \int_\R \abs{(u^S_\Y)'} \abs{\p_\xi\Phi} \abs{\p_\xi \Wn} d\xi dt + C\e \int_0^T e^{-\alpha t} \norm{\p_\xi\Phi} \norm{\p_\xi \Wn} dt \\
		& \leq \frac{a_0}{2} \int_0^T \norm{\sqrt{-(u^S_\Y)'} \p_\xi\Wn}^2 dt + C \int_{0}^{T} \norm{\p_\xi \Phi}^2 dt + C\e \sup_{t \in (0,T)} \norm{\p_\xi\Wn}^2,
	\end{align*}
	where $ a_0 = \inf \frac{p''(\vtn)}{2(p'(\vtn))^2} >0. $
	By \cref{nonlinear-p,Q}, one has that
	$ \abs{\p_\xi p(\vn|\vtn)} \leq C \left( \abs{\p_\xi \Phi}^2 + \abs{\p_\xi \Phi} \abs{\p_\xi^2 \Phi} \right) $ and $ \abs{\Qn} \leq C \left( \abs{\p_\xi\Phi} + \abs{\p_\xi\Phi}\abs{\p_\xi^2 \Phi} \right). $ Thus, it follows from \cref{bdd-apriori} and Sobolev inequality that
	\begin{align*}
	\int_0^T \int_\R \abs{I_{2,4}} d\xi dt & \leq C\int_0^T \int_\R  \left( \abs{\p_\xi\Phi} +  \norm{\p_\xi \Phi}_{L^\infty(\R)}\abs{\p_\xi^2 \Phi} \right)  \abs{\p_\xi^2 \Phi} d\xi dt \\
	& \quad + C\int_0^T \int_\R \norm{\p_\xi  \Phi}_{L^\infty(\R)} \left( \abs{\p_\xi\Phi} + \abs{\p_\xi^2 \Phi} \right) \abs{\p_\xi \Wn} d\xi dt \\
	& \leq C \int_{0}^{T} \norm{\p_\xi\Phi} \norm{\p_\xi^2 \Phi} dt + C\delta \int_{0}^{T} \norm{\p_\xi^2 \Phi}^2 dt \\
	& \quad +C \sup_{t \in (0,T)} \norm{\p_\xi\Wn}  \int_0^T \norm{\p_\xi\Phi}^{\frac{1}{2}} \norm{\p_\xi^2\Phi}^{\frac{1}{2}} \big(\norm{\p_\xi\Phi} +\norm{\p_\xi^2\Phi} \big) dt\\
	& \leq C \int_{0}^{T} \norm{\p_\xi\Phi}^2 dt + \Big(\frac{b_0}{4}+C\delta \Big) \int_{0}^{T} \norm{\p_\xi^2 \Phi}^2 dt,
	\end{align*}
	where $ b_0 = \inf \sigma'(\vtn)>0. $
	Similar to \cref{I-1-2}, \cref{Lem-H} yields that 
	\begin{align*}
		\int_0^T \int_\R \abs{I_{2,5}} d\xi dt & \leq C \int_0^T \big[ \norm{\Hn_1} \norm{\p_\xi^2 \Phi} + \big(\norm{\p_\xi \Hn_2}+\norm{\p_\xi \Hn_1}_1 \big) \norm{\p_\xi  \Wn} \big] dt \\
		& \leq C\e\int_0^T \norm{\p_\xi^2 \Phi}^2 dt + C\e + C\e \sup_{t\in(0,T)} \norm{\p_\xi  \Wn}^2. 
	\end{align*}
	Thus, collecting the estimates from $ I_{1,1} $ and $ I_{2,2} $ to $ I_{2,5}, $ one can integrate \cref{equ-2} over $ \R\times(0,T) $ to get that
	\begin{equation*}
	\begin{aligned}
	& \sup_{t \in (0,T)} \norm{\p_\xi (\Phi,\Wn)}^2 + \int_{0}^{T} \Big(\norm{\p_\xi^2 \Phi}^2 + \norm{\sqrt{-(u^S_\Y)'}~ \p_\xi\Wn }^2 \Big) dt \\
	& \qquad \leq C \Big( \norm{\Phi_0, \Wn_0}_1^2 + \e + \int_{0}^{T} \norm{\p_\xi \Phi}^2 dt \Big),
	\end{aligned}  
	\end{equation*}
	if $ \delta_0 $ and $ \e_0 $ are small.
	This, together with \cref{Lem-esti-0}, yields \cref{esti-1}.
\end{proof}

\vspace{0.3cm}

Now we go back to \cref{equ-Phi-Psi} to estimate $ \p_\xi^2 \Phi $ and the derivatives of $ \Psi. $

\begin{Lem}\label{Lem-Psi}
Under the assumptions of \cref{Prop-apriori},  there exist $ \e_0>0 $ and $ \delta_0>0 $ such that if $ \e<\e_0 $ and $ \delta<\delta_0, $ then
\begin{equation*}
\begin{aligned}
\sup_{t \in (0,T)} \big(\norm{\p_\xi^2\Phi}^2 + \norm{\Psi}_2^2\big)(t) + \int_0^T \norm{\p_\xi \Psi}_2^2(t) dt \leq C (\norm{\Phi_0,\Psi_0}_2^2 + \e),
\end{aligned}
\end{equation*}
where $ C>0 $ is independent of $ \e,\delta $ and $ T. $
\end{Lem}

\begin{proof}
First, it follows from Lemmas \ref{Lem-esti-0} and \ref{Lem-esti-1} that
\begin{align}
& \sup_{t \in (0,T)} \big(\norm{\Phi}_1^2 + \norm{\Psi}^2\big) + \int_0^T \Big( \norm{\p_\xi\Phi}_1^2 + \norm{\sqrt{-(u^S_\Y)'} (\Psi, \p_\xi \Psi)}^2 \Big) dt \notag \\
& \qquad \leq C \big(\norm{\Phi_0,\Wn_0}_1^2 + \e\big)  \leq C \big(\norm{\Phi_0}_2^2 + \norm{\Psi_0}_1^2 + \e\big).  \label{ineq-Psi-0}
\end{align}

1) To estimate $ \int_0^T \norm{\p_\xi \Psi}^2 dt, $ by multiplying $ \Phi $ on \cref{equ-Phi-Psi}$ _1 $ and $ -\frac{1}{p'(\vtn)}\Psi $ on \cref{equ-Phi-Psi}$ _2, $ respectively, one can get that
	\begin{equation}\label{eq-est-1}
		\begin{aligned}
			& \p_t\Big(\frac{\abs{\Phi}^2}{2} + \frac{\Psi^2}{2\abs{p'(\vtn)}}\Big) - \frac{p''(\vtn) (u^S_\Y)'}{2(p'(\vtn))^2} \Psi^2 + \frac{\sigma'(\vn)}{\abs{p'(\vtn)}} \abs{\p_\xi \Psi}^2 \\
			& \quad = I_{1,1} \Psi^2 + \underbrace{\frac{p(\vn|\vtn)}{p'(\vtn)} \Psi + \frac{\sigma''(\vn)}{p'(\vtn)} \p_\xi^2 \Phi \Psi \p_\xi \Psi}_{I_{3,1}} \\
			& \qquad  \underbrace{+ \Big( \frac{\sigma''(\vn)}{p'(\vtn)} - \frac{\sigma'(\vn)p''(\vtn)}{(p'(\vtn))^2}\Big)  \p_\xi \vtn \Psi \p_\xi\Psi - \frac{1}{p'(\vtn)} \big(\sigma'(\vn)-\sigma'(\vtn)\big)\p_\xi \utn \Psi}_{I_{3,2}} \\
			& \qquad \underbrace{+ \Hn_1 \Phi - \frac{1}{p'(\vtn)} \Hn_2 \Psi}_{I_{3,3}} + \p_\xi \{\cdots\},
		\end{aligned}	
	\end{equation}
	where $ \{\cdots\} = \frac{s}{2} \Phi^2 + \Phi \Psi - \frac{s}{2p'(\vtn)} \Psi^2 - \frac{\sigma'(\vn)}{p'(\vtn)} \Psi \p_\xi \Psi.  $
	The nonlinear term $ I_{3,1} $ satisfies that
	\begin{align*}
		\int_0^T \int_\R \abs{I_{3,1}} d\xi dt & \leq C \norm{\Psi}_{L^\infty} \int_0^T (\norm{\p_\xi\Phi}^2+\norm{\p_\xi^2 \Phi} \norm{\p_\xi\Psi}) dy \\
		& \leq C \delta \int_0^T (\norm{\p_\xi \Phi}_1^2+\norm{\p_\xi \Psi}^2) dt.
	\end{align*}
	For $ I_{3,2}, $ it follows from \cref{term-1,term-2} that
	\begin{align*}
		\int_0^T \int_\R \abs{I_{3,2}} d\xi dt & \leq C \int_0^T \int_\R \abs{(v^S_\X)'} \abs{\Psi} \abs{\p_\xi\Psi} d\xi dt + C\e \int_0^T e^{-\alpha t} \norm{\Psi} \norm{\p_\xi \Psi} dt \\
		& \quad + C \int_0^T \int_\R \abs{(u^S_\Y)'} \abs{\p_\xi \Phi} \abs{\Psi} d\xi dt  + C\e \int_0^T e^{-\alpha t} \norm{\p_\xi \Phi} \norm{\Psi} dt \\
		& \leq C \int_0^T \norm{\sqrt{(-u^S_\Y)'} (\Psi, \p_\xi \Psi)}^2 dt + C \e \sup_{t \in (0,T)} \norm{\Psi}^2 \\
		& \quad + C\e \int_0^T \norm{\p_\xi \Psi}^2 dt + C \int_0^T \norm{\p_\xi \Phi}^2 dt.
	\end{align*}
	And similar to \cref{I-1-2}, it follows from \cref{Lem-H} that
	\begin{align*}
		\int_0^T \int_\R \abs{I_{3,3}} d\xi dt & \leq C\e \sup_{t \in (0,T)} \norm{\Phi,\Psi}^2 + C\e.
	\end{align*}
	Collecting the estimates from  $ I_{3,1} $ to $ I_{3,3}, $ and using \cref{I-1-1,ineq-Psi-0}, one can integrate \cref{eq-est-1} over $ \R\times(0,T) $ to get that
	\begin{equation}\label{est-Psi-1}
		\int_0^T \norm{\p_\xi \Psi}^2 dt	\leq C \norm{\Phi_0}_2^2 + C\norm{\Psi_0}_1^2 + C\e.	
	\end{equation} 

2) To estimate $ \psi, $ we differentiate \cref{equ-Phi-Psi}$ _2 $ with respect to $ \xi $ once and then multiply $ \psi $ on the resulting equation, which gives that
\begin{equation}\label{eq-4}
	\begin{aligned}
		& \p_t \Big(\frac{\psi^2}{2} \Big) + \sigma'(\vn) \big(\p_\xi \psi \big)^2 = \big[ p(\vn) - p(\vtn)\big] \p_\xi \psi \\
		& \qquad\qquad\qquad\qquad\qquad\quad - (\sigma'(\vn)-\sigma'(\vtn)) \p_\xi \utn \p_\xi \psi - \Hn_2 \p_\xi \psi +\p_\xi \{\cdots\},
	\end{aligned}
\end{equation}
where $ \{\cdots\} = \frac{s}{2}\psi^2 - \big[ p(\vn) - p(\vtn)\big]\psi + \big[ \sigma'(\vn) \p_\xi \un-\sigma'(\vtn) \p_\xi \utn \big] \psi + \Hn_2\psi. $ 
Integrating \cref{eq-4} over $ \R\times(0,T) $ yields that
\begin{equation*}
	\begin{aligned}
		\sup_{t \in (0,T)} \norm{\psi}^2 + \int_0^T \norm{\p_\xi \psi}^2 dt & \leq C \norm{\Psi_0}_1^2 + C \int_0^T \big(\norm{\phi} +  \norm{\Hn_2}\big) \norm{\p_\xi\psi} dt \\
		& \leq C \norm{\Psi_0}_1^2 + \frac{1}{2} \int_0^T \norm{\p_\xi \psi}^2 dt + C \int_0^T \norm{\phi}^2 dt + C\e.
	\end{aligned}
\end{equation*}
This, together with \cref{ineq-Psi-0}, yields that
\begin{equation}\label{ineq-psi}
	\sup_{t \in (0,T)} \norm{\psi}^2 + \int_0^T \norm{\p_\xi \psi}^2 dt \leq C (\norm{\Phi_0}_2^2 + \norm{\Psi_0}_1^2 + \e).
\end{equation}

3) To estimate $ \p_\xi \psi, $ we differentiate \cref{eq-4}$ _2 $ with respect to $ \xi $ twice and multiply $ \p_\xi \psi $ on the resulting equation, which gives that
\begin{equation}\label{eq-6}
	\begin{aligned}
		& \p_t \Big(\frac{\abs{\p_\xi \psi}^2}{2} \Big) + \sigma'(\vn) \abs{\p_\xi^2 \psi}^2 = \p_\xi\big[p(\vn)-p(\vtn) \big] \p_\xi^2 \psi  - \sigma''(\vn) \p_\xi \vtn \p_\xi \psi \p_\xi^2 \psi \\
		& \quad - \sigma''(\vn) \p_\xi \phi \p_\xi \psi \p_\xi^2 \psi - \p_\xi \big[(\sigma'(\vn)-\sigma'(\vtn)) \p_\xi \utn \big] \p_\xi^2 \psi -\p_\xi \Hn_2 \p_\xi^2 \psi + \p_\xi \{\cdots\},
	\end{aligned}
\end{equation}
where $ \{\cdots\} = \frac{s}{2} \abs{\p_\xi\psi}^2 - \p_\xi \big[p(\vn)-p(\vtn) \big] \p_\xi\psi + \p_\xi \big[ \sigma'(\vn) \p_\xi \un-\sigma'(\vtn) \p_\xi \utn \big] \p_\xi \psi + \p_\xi \Hn_2 \p_\xi \psi. $ By \cref{assum-apriori} and Sobolev inequality, the most difficult term $ \sigma''(\vn) \p_\xi \phi \p_\xi\psi \p_\xi^2\psi $ satisfies that
\begin{align*}
	\int_0^T \int_\R \abs{\sigma''(\vn) \p_\xi \phi \p_\xi\psi \p_\xi^2\psi } d\xi dt & \leq C \int_0^T \norm{\p_\xi\phi} \norm{\p_\xi\psi}_{L^\infty(\R)} \norm{\p_\xi^2 \psi} dt \\
	& \leq C \sup_{t \in (0,T)} \norm{\p_\xi\phi}^{\frac{1}{2}} \norm{\p_\xi\psi}^{\frac{1}{2}} \int_0^T \norm{\p_\xi\phi}^{\frac{1}{2}} \norm{\p_\xi^2 \psi}^{\frac{3}{2}} dt \\
	& \leq C \delta \int_0^T \big(\norm{\p_\xi\phi}^2 + \norm{\p_\xi^2 \psi}^2 \big) dt.
\end{align*}
Then integrating \cref{eq-6} over $ \R\times(0,T) $ yields that
\begin{align*}
	\sup_{t \in (0,T)} \norm{\p_\xi\psi}^2 + \int_0^T \norm{\p_\xi^2\psi}^2 dt & \leq C \Big(\norm{\Psi_0}_2^2 + \int_0^T \norm{\phi}_1^2 dt + \int_0^T \norm{\p_\xi\psi}^2 dt + \int_0^T \norm{\p_\xi\Hn_2}^2 dt\Big),
\end{align*}
which, together with \cref{ineq-Psi-0,ineq-psi}, implies that
\begin{equation}\label{est-Psi-2}
	\sup_{t \in (0,T)} \norm{\p_\xi\psi}^2 + \int_0^T \norm{\p_\xi^2\psi}^2 dt \leq C \big(\norm{\Phi_0}_2^2 + \norm{\Psi_0}_2^2 + \e \big).
\end{equation}

4) Finally, we estimate $ \p_\xi \phi. $ Differentiating  \cref{equ-Phi-Psi}$ _1 $ with respect to $ \xi $ twice and multiplying $ \p_\xi \phi $ on the resulting equation gives that
\begin{equation}\label{eq-5}
	\begin{aligned}
		& \p_t \Big(\frac{\abs{\p_\xi \phi}^2}{2}\Big) = \p_\xi \phi \p_\xi^2 \psi + \p_\xi\phi \p_\xi^2\Hn_1 + \p_\xi \Big( \frac{s\abs{\p_\xi \phi}^2}{2}\Big).
	\end{aligned}
\end{equation}
Thus, it follows from \cref{ineq-Psi-0,est-Psi-2} that
\begin{align*}
	\sup_{t \in (0,T)} \norm{\p_\xi\phi}^2 & \leq C \big(\norm{\Phi_0}_2^2 + \int_0^T \norm{\phi}_1^2 dt + \int_0^T \norm{\p_\xi^2 \psi}^2 dt + \int_0^T \norm{\p_\xi^2\Hn_1}^2 dt \big) \\
	& \leq C (\norm{\Phi_0}_2^2 + \norm{\Psi_0}_2^2 + \e).
\end{align*}

Collecting the estimates from 1) to 4), the proof of \cref{Lem-Psi} is completed.
\end{proof}

Then \cref{Prop-apriori} can follow immediately from Lemmas \ref{Lem-esti-1}--\ref{Lem-Psi}.


\subsection{Proof of \cref{Thm-pert}}

The local existence of the solution of \cref{equ-Phi-Psi,ic-anti} in $ \mathcal{B}(0,T_0) $ can be proved by the standard contraction mapping theorem. 
Hence, with \cref{Prop-apriori}, one can let $ \e_0>0 $ be small such that $ \sup\limits_{t \in (0,T_0)} \norm{\Phi,\Psi}_2^2 \leq C_0(\e_0^2 + \e_0) < \delta_0^2, $ then the a priori assumptions \cref{assum-apriori} can be closed.
Through a standard continuation argument, one can obtain a global in time solution $ \left(\Phi, \Psi \right)\in \mathcal{B}(0,+\infty). $

Once the estimate \cref{apriori} with $ T=+\infty $ is obtained, one can follow the same line as in \cite{Matsumura1985} to prove that $ \norm{\phi,\psi}_{L^\infty(\R)}(t) \to 0 $ as $ t\to+\infty, $ i.e. \cref{0limit} holds true.  The details are omitted. Thus, the proof of \cref{Thm-pert} is completed.


\section{Proof of Lemmas \ref{Lem-shift} and \ref{Lem-H}}\label{Sec-prop-shift}

For convenience, let
\begin{equation}\label{phi-psi-lr}
	\begin{aligned}
		& \phi_{l,r}(x,t) := v_{l,r}(x,t)-\ovl{v}_{l,r},\quad \psi_{l,r}(x,t) := u_{l,r}(x,t)-\ovl{u}_{l,r},
	\end{aligned}
\end{equation} 
all of which are space-periodic functions with zero averages for any $ t\geq 0. $

\subsection{Proof of \cref{Lem-shift}}

Since $ \e_0 $ is small enough, it follows from \cref{Lem-periodic} that the denominators of \cref{ode-shift} are away from zero and then the right-hand side of \cref{ode-shift} are $ C^1 $ smooth. Then the existence and uniqueness of the solution $ (\X, \Y) \in C^1\left([0,+\infty)\right) $ to \cref{ode-shift} with the initial data $ (\X,\Y)(0)=(\X_0, \Y_0) $ can follow from the Cauchy-Lipschitz theorem.
And from \cref{R-H} and \cref{Lem-periodic}, one can easily prove that
$ \abs{\X'(t)}, \abs{\Y'(t)} \leq C\e e^{-\alpha t} $ for all $ t>0. $ Thus both limits, $ \lim\limits_{t\to +\infty} \X(t) $ and $ \lim\limits_{t\to +\infty} \Y(t), $ exist and it holds that $ \abs{\X(t)-\X_0} + \abs{\Y(t)-\Y_0} \leq C\e. $ 

Motivated by \cite{Xin2019}, we first compute $ \lim\limits_{t\to +\infty} \Y(t). $
For any fixed $ x\in [0,1], t>0, $ and integer $ N>0, $ define the domain (see Figure 1)
\begin{figure}[htbp!]
	\centering	\scalebox{0.25}{\includegraphics{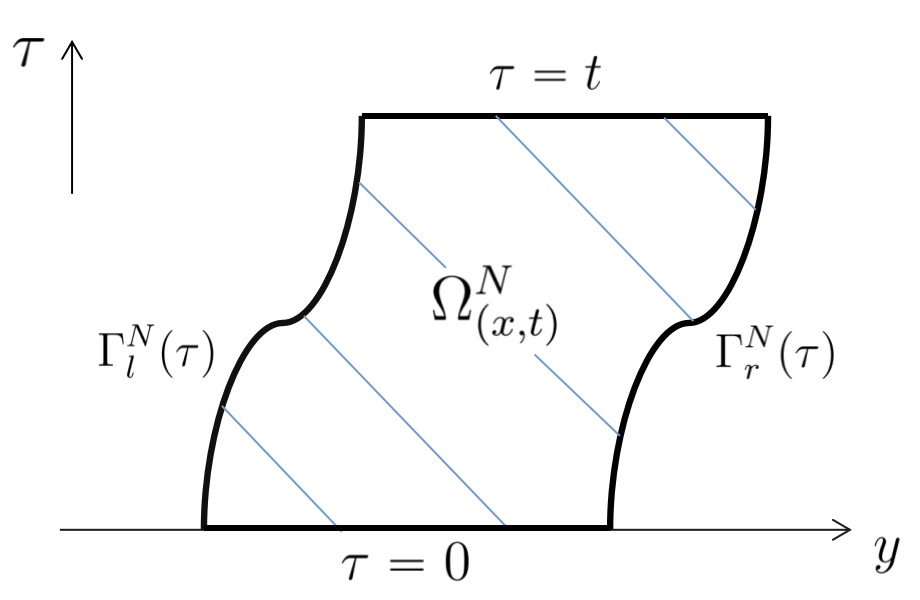}}
	\caption{Domain $ \Omega_{(x,t)}^N $}
\end{figure}
\begin{equation*}
	\begin{aligned}
		\Omega_{(x,t)}^N & := \big\{ (y,\tau): 0<\tau<t,\quad \Gamma_l^N(\tau) < y < \Gamma_r^N(\tau) \big\} \\
		\text{with } \quad \Gamma_l^N(\tau
		) & := s\tau+\Y(\tau)+(-N+x) \pi_l, \\
		\Gamma_r^N(\tau) & := s\tau+\Y(\tau) +(N+x)\pi_r.
	\end{aligned}
\end{equation*}
In the following, let $ C>0 $ be a generic constant independent of $ \e, N $ and $ t, $ and use $ O(1) $ to denote the terms which can be bounded by $ C. $ 

Integrating \cref{eq-ansatz}$ _2 $ over $ \Omega_{(x,t)}^N $ yields that
\begin{equation}\label{iden-1}
	\begin{aligned}
		&\int_{\Gamma_l^N(0)}^{\Gamma_r^N(0)} \ut(y,0) dy + \int_0^t \left[ -p(\vt) + \sigma'(\vt) \p_x\ut + \ut \left(s+\Y'\right) \right]\left(\Gamma_r^N(\tau),\tau\right)d\tau \\
		& \qquad -\int_{\Gamma_l^N(t)}^{\Gamma_r^N(t)} \ut(y,t) dy - \int_0^t \left[ -p(\vt) + \sigma'(\vt) \p_x\ut + \ut \left(s+\Y'\right) \right] \left(\Gamma_l^N(\tau),\tau\right) d\tau \\
		& \quad = -\iint_{\Omega_{(x,t)}^N} \left(\p_y F_3 + f_4 \right) dyd\tau
	\end{aligned}
\end{equation}
By the definition of $ \ut $ in \cref{ansatz-shock}, one has that
	\begin{align}
	& \int_{\Gamma_l^N(0)}^{\Gamma_r^N(0)} \ut(y,0) dy -\int_{\Gamma_l^N(t)}^{\Gamma_r^N(t)} \ut(y,t) dy \notag \\
	& \quad = \int_{\Gamma_l^N(0)}^{\Gamma_r^N(0)} \left[ \psi_{0l}(1-g_{\Y_0}) + \psi_{0r} g_{\Y_0} \right] dy - \int_{\Gamma_l^N(t)}^{\Gamma_r^N(t)} \left[\psi_l(1-g_{\Y}) + \psi_r g_{\Y} \right] dy \notag \\
	& \quad = \int_{0}^{\Gamma_r^N(0)} \left(\psi_{0l}-\psi_{0r}\right)(1-g_{\Y_0}) dy + \int_{0}^{\Y_0+ x\pi_r} \psi_{0r}(y)dy \notag \\
	& \qquad  + \int_{\Gamma_l^N(0)}^0 (\psi_{0r}-\psi_{0l}) g_{\Y_0} dy + \int_{\Y_0+x \pi_l}^0 \psi_{0l}(y) dy - R_1, \label{eq-7}
\end{align}
where 
\begin{align*}
	R_1(x,t) = & \int_{0}^{\Gamma_l^N(t)} \left(\psi_l-\psi_r\right)(1-g_{\Y}) dy + \int_{0}^{st+\Y(t)+ x\pi_r} \psi_r(y,t)dy \\
	&~ + \int_{\Gamma_l^N(t)}^0 (\psi_r-\psi_l) g_{\Y} dy + \int_{st+\Y(t)+x \pi_l}^0 \psi_l(y,t) dy,
\end{align*}
which satisfies $ \norm{R_1}_{L^\infty(\R)} \leq C \left(\norm{\psi_l}_{L^\infty(\R)}+ \norm{\psi_r}_{L^\infty(\R)} \right)
\leq Ce^{-\alpha t}. $
And the second integral on the left hand side of \cref{iden-1} satisfies 
\begin{align}
	& \int_0^t \left[ -p(\vt) + \sigma'(\vt) \p_x\ut + \ut \left(s+\Y'\right) \right] \left(\Gamma_r^N(\tau),\tau\right) d\tau \notag \\
	& \quad = \int_{0}^{t} \left[ -p(v_r) + \sigma'(v_r) \p_x u_r + u_r \left(s+\Y'\right) \right] \left(\Gamma_r^N(\tau),\tau\right) d\tau + R_2 \notag \\
	& \quad = \int_{0}^{t} \left[ -p(v_r) + \sigma'(v_r) \p_t v_r + u_r \left(s+\Y'\right) \right] \left(\Gamma_r^N(\tau),\tau\right) d\tau  + R_2 \notag \\
	& \quad = \int_{0}^{t} \left[ -p(v_r) + u_r \left(s+\Y'\right) \right] \left(s\tau+\Y(\tau) + x\pi_r,\tau\right)d \tau \notag \\
	& \qquad  + \sigma\big(v_r(st+\Y(t)+x\pi_r,t)\big) - \sigma\big(\ovlvr+\phi_{0r}(\Y_0+x\pi_r) \big) + R_2, \label{eq-8}
\end{align}
where 
\begin{align*}
	& R_2 = \int_0^t \big\{ - \left[p(\vt)-p(v_r)\right] + \left[\sigma'(\vt) \p_x\ut -\sigma'(v_r) \p_x u_r \right] + \left(\ut-u_r\right) \left(s+\Y'\right) \big\}\left(\Gamma_r^N(\tau),\tau\right)d\tau.
\end{align*}
Note that
\begin{align*}
	\sigma'(\vt) \p_x \ut - \sigma'(v_r) \p_x u_r & = \sigma'(\vt) \left[ (u_r-u_l) g'_{st+\Y} + \p_x u_l (1-g_{st+\Y}) + \p_x u_r g_{st+\Y} \right] - \sigma'(v_r) \p_x u_r \\
	& = O(1) g'_{st+\Y} + O(1) (1-g_{st+\Y}) + \p_x u_r \left( \sigma'(\vt) - \sigma'(v_r) \right) \\
	& = O(1) g'_{st+\Y} + O(1) (1-g_{st+\Y}) + O(1) (1-g_{st+\X}).
\end{align*}
Thus, it holds that
\begin{align*}
	\abs{R_2} \leq C \int_{0}^{t} \big\{ & \abs{ 1-g\left(\Y(\tau)-\X(\tau) + (N+x)\pi_r\right)} + \abs{ g'\left((N+x)\pi_r\right)} +\abs{ 1-g\left((N+x)\pi_r \right)}  \big\} d\tau,
\end{align*}
which tends to zero as $ N\to +\infty $ for fixed $ t>0. $
Similarly, one can prove that
\begin{align}
	& \int_0^t \left[-p(\vt) + \sigma'(\vt) \p_x\ut + \ut \left(s+\Y'\right) \right] \left(\Gamma_l^N(\tau),\tau\right) d\tau \notag \\
	& \quad = \int_{0}^{t} \left[ -p(v_l) + u_l \left(s+\Y'\right) \right] \left(s\tau+\Y(\tau)+x\pi_l, \tau\right) d\tau \notag \\
	& \qquad  + \sigma\big(v_l(st+\Y(t)+x\pi_l,t)\big) - \sigma\big(\ovlvl+\phi_{0l}(\Y_0+x\pi_l)\big) + R_3, \label{eq-9}
\end{align}
where the remaining term $ R_3 $ 
satisfies that $ \abs{R_3} \to 0 $ as $ N\to +\infty. $

Collecting \cref{eq-7,eq-8,eq-9}, by integrating \cref{iden-1} with respective to $ x $ over $ [0,1] $ and then letting $ N\rightarrow +\infty, $ one can get that 
\begin{align*}
	0 =
	& \int_{0}^{+\infty} \left(\psi_{0l}-\psi_{0r}\right)(1-g_{\Y_0}) dy + \frac{1}{\pi_r} \int_{0}^{\pi_r}\int_{0}^{x} \psi_{0r}(y)dy dx  - \int_{-\infty}^0 (\psi_{0l}-\psi_{0r}) g_{\Y_0} dy \\
	& - \frac{1}{\pi_l} \int_{0}^{\pi_l}\int_{0}^{x} \psi_{0l}(y)dy dx - \int_{0}^{t} \frac{1}{\pi_r} \int_{0}^{\pi_r} \left[ p(v_r(x,\tau)) - p(\ovlvr) \right] dx d\tau - p(\ovlvr) t \\
	& + \ovlur \left(st+\Y(t) - \Y_0\right) 
	+ \sigma(\ovlvr) - \frac{1}{\pi_r} \int_{0}^{\pi_r} \sigma(\ovlvr+\phi_{0r}(x)) dx \\
	& + \int_{0}^{t} \frac{1}{\pi_l} \int_{0}^{\pi_l} \left[p(v_l(x,\tau))-p(\ovlvl) \right] dx d\tau + p(\ovlvl) t - \ovlul \left(st+\Y(t)-\Y_0\right) \\
	& - \sigma(\ovlvl) + \frac{1}{\pi_l} \int_{0}^{\pi_l} \sigma(\ovlvl+\phi_{0l}(x)) dx  + O(1)\e e^{-\alpha t} \\
	= & (\ovlur-\ovlul) \left(\Y(t) -\Y_\infty \right) + O(1)\e e^{-\alpha t}.
\end{align*}
In the same way, one can prove $ \lim\limits_{t\to+\infty} \X(t) = \X_\infty. $ The proof of \cref{Lem-shift} is completed.


\subsection{Proof of \cref{Lem-H}}\label{Sec-H}	
	The idea is to extract the ``averages'', $ \ovlvl,\ovlul,\ovlvr,\ovlur,v_\X^S $ and $ u_\Y^S, $ from $ \vn_l,\un_l,\vn_r,\un_r,\vtn $ and $ \utn, $ respectively, where all the differences decay exponentially fast with respect to $ t, $ e.g. $ \abs{\vtn-v^S_\X}=\abs{(\vn_l-\ovlvl) (1-g_\X) + (\vn_r-\ovlvr) g_\X} \leq C\e e^{-\alpha t}. $

1) We first estimate $ \Hn_1. $ From \cref{source}, we decompose $ \Hn_1 $ for $ \xi<0 $ as
	\begin{equation}\label{dec-H-1}
		\Hn_1(\xi,t) = \Fn_1(\xi,t) + \int_{-\infty}^{\xi} \fn_2(y,t)dy = D_{1,1}^-(\xi,t) + D_{1,2}^-(\xi,t),
	\end{equation}
	where 
	\begin{align*}
		D_{1,1}^-(\xi,t) := & \left(\ovlur-\ovlul\right) \left(g_\Y-g_\X\right)(\xi) + (\ovlur-\ovlul)g_{\X}(\xi) +  (s+\X')(\ovlvr-\ovlvl) g_{\X}(\xi) \\
		= & (\ovlvr-\ovlvl) \left[ \X'(t) g_\X(\xi) + s \left(g_\X-g_\Y\right)(\xi) \right], \\
		D_{1,2}^-(\xi,t) := & \left(\phi_r-\phi_l\right)(\xi+st,t) \left(g_\Y-g_\X\right)(\xi) + \int_{-\infty}^\xi (\psi_r-\psi_l)(y+st,t) g'_{\X}(y) dy \\
		& + (s+\X') \int_{-\infty}^\xi (\phi_r-\phi_l)(y+st,t) g'_{\X}(y) dy.
	\end{align*}
	From \cref{Lem-periodic}, one has that
\begin{equation*}
	\sum_{k=0}^{2}\norm{\p_\xi^k \left( \phi_l,\psi_l,\phi_r,\psi_r \right)}_{L^\infty(\R)} \leq C \e e^{-\alpha t}, \quad t>0.
\end{equation*}
Hence, by the fact that 
$$ \left(g_\X-g_\Y\right)(\xi) = \int_0^1 g'\left(\xi-\X+\theta\left( \Y-\X\right)\right) d\theta \left( \Y-\X\right) $$ with $ \abs{\Y-\X}(t) \leq C \e e^{-\alpha t}, $ one can get that
\begin{equation*}
	\begin{aligned}
	\sum_{k=0}^2 \int_{-\infty}^0 \abs{\p_\xi^k \Hn_1}^2 d\xi & \leq \sum_{k=0}^2  \int_{-\infty}^0 \left(\abs{\p_\xi^k D_{1,1}^-}^2 + \abs{\p_\xi^k D_{1,2}^-}^2 \right) d\xi \\
		& \leq C \e^2 e^{-2\alpha t} \sum_{k=0}^{3} \int_{-\infty}^{M_0} \abs{\frac{d^k}{d\xi^k} g(\xi)}^2 d\xi \leq C \e^2 e^{-2\alpha t}.
	\end{aligned}
\end{equation*}
where $ M_0 := \sup\limits_{t\geq 0} \left(\abs{\X} + \abs{\Y}\right)(t). $ 
In a similar way, one can decompose $ \Hn_1 $ for $ \xi >0 $ through
	\begin{equation*}
		\Hn_1 = \Fn_1(\xi,t) - \int_{\xi}^{+\infty} \fn_2(y,t)dy = D_{1,1}^+(\xi,t) + D_{1,2}^+(\xi,t),
	\end{equation*}
	where it holds that
	\begin{align}
	\sum_{k=0}^{2} \int_0^{+\infty} \abs{\p_\xi^k \Hn_1}^2 d\xi & \leq \sum_{k=0}^{2} \int_0^{+\infty} \left(\abs{\p_\xi^k D_{1,1}^+}^2 + \abs{\p_\xi^k D_{1,2}^+}^2\right) d\xi \notag \\ 
	& \leq C \e^2 e^{-2\alpha t} \sum_{k=0}^{3} \int_{-M_0}^{+ \infty} \abs{\frac{d^k}{d\xi^k} (1-g(\xi))}^2 d\xi \leq C \e^2 e^{-2\alpha t}. \label{H-1-plus}
	\end{align}

2) For $ \Hn_2 = -\sigma'(\vtn) \p_\xi \Hn_1 + \Fn_3 + \Fn_4, $ it follows from \cref{source} that for $ \xi<0, $
	\begin{equation*}
		\begin{aligned}
			\Fn_3 + \Fn_4 = & ~\Fn_3 + \int_{-\infty}^{\xi} \fn_4(y,t)dy \\
			= &  -\left(p(\vtn)-p(\vn_l)\right) + \left(p(\vn_r)-p(\vn_l)\right) g_\Y + \sigma'(\vtn) (\un_r-\un_l) g'_\Y  \\
			& + \p_\xi \un_l  \left(\sigma'(\vtn)-\sigma'(\vn_l)\right)(1-g_\Y) + \p_\xi \un_r \left( \sigma'(\vtn) - \sigma'(\vn_r)\right)g_\Y \\
			& + (s+\Y') \int_{-\infty}^{\xi} (\un_r-\un_l) g'_\Y dy + \int_{-\infty}^{\xi} \big(\sigma'(\vn_r) \p_y \un_r - \sigma'(\vn_l)\p_y \un_l \big) g'_\Y dy \\
			& - \int_{-\infty}^{\xi} \left(p(\vn_r) - p(\vn_l) \right) g'_\Y dy.
		\end{aligned}
	\end{equation*}
	Note that $ \sigma'(\vtn)-\sigma'(\vn_l) = b_1(\vn_l,\vtn) \left(\vtn-\vn_l\right) = b_1(\vn_l,\vtn) \left(\vn_r-\vn_l\right) g_\X, $ where $ b_1(u,v) := \int_{0}^{1} \sigma''(u+\theta(v-u)) d\theta. $ And similarly, $ \sigma'(\vtn)-\sigma'(\vn_r) = - b_1(\vn_r,\vtn) \left(\vn_r-\vn_l\right) (1-g_\X). $
	Thus, we decompose $ \Fn_3+\Fn_4 = D_{2,1}^- + D_{2,2}^- $ as follows,
	\begin{equation*}
		\begin{aligned}
			D_{2,1}^- = & - p(v^S_\X) + p(\ovlvl) + \left(p(\ovlvr) - p(\ovlvl)\right)g_\Y + \sigma'(v_\X^S) \left(\ovlur-\ovlul\right) g'_\Y \\
			& + \left(s+\Y'\right) \left(\ovlur-\ovlul\right) g_\Y - \left( p(\ovlvr)-p(\ovlvl) \right) g_\Y, \\
			D_{2,2}^- = &  - \left( p(\vtn) - p(v^S_\X) - p(\vn_l) + p(\ovlvl) \right) \\
			& - \left[\left(p(\vn_r) -p(\ovlvr)\right) - \left(p(\vn_l)- p(\ovlvl)\right)\right] g_{\Y} \\
			& + \left[\sigma'(\vtn) (\un_r-\un_l) - \sigma'(v_\X^S) \left(\ovlur-\ovlul\right) \right]  g'_\Y  \\
			& + \left(\vn_r-\vn_l\right) \left[ \p_\xi \un_l ~b_1(\vn_l,\vtn) g_\X (1-g_{\Y}) - \p_\xi \un_r~ b_1(\vn_r,\vtn) (1-g_{\X})g_{\Y} \right] \\
			& + (s+\Y') \int_{-\infty}^{\xi} (\psi_r-\psi_l)(y+st,t) g'_{\Y}(y) dy \\
			& + \int_{-\infty}^{\xi} \left(\sigma'(\vn_r) \p_y \un_r - \sigma'(\vn_l)\p_y \un_l \right) g'_{\Y}(y) dy \\
			& - \int_{-\infty}^{\xi} \left(p(\vn_r) - p(\ovlvr) - p(\vn_l) + p(\ovlvl) \right) g'_{\Y} dy.
		\end{aligned}
	\end{equation*}
	Note that $ (\ovlur-\ovlul) g_\Y = u^S_\Y - \ovlul, $ then it follows from \cref{ode-1} that
	\begin{align*}
		D_{2,1}^- = p(v^S_\Y)- p(v^S_\X)  + \left(\sigma'(v_\X^S)-\sigma'(v_\Y^S)\right) \left(\ovlur-\ovlul\right) g'_\Y  +  \Y' \left(\ovlur-\ovlul\right) g_\Y.
	\end{align*}	
	Denote $ b_2(u,v) := \int_{0}^{1} p'(u+\theta(v-u)) d\theta. $ Then the first term of $ D_{2,2}^- $ satisfies that
	\begin{equation*}
		\begin{aligned}
			& -p(\vtn)+ p(\vn_l) + p(v^S_\X) - p(\ovlvl) = -\left[b_2(\vn_l,\vtn) \left(\vn_r-\vn_l\right) - b_2(\ovlvl,v_\X^S) \left(\ovlvr-\ovlvl\right)\right] g_\X	\\
			& \qquad = \left[\left(b_2(\vn_l,\vtn)-b_2(\ovlvl,v_\X^S)\right)\left(\vn_r-\vn_l\right) + b_2(\ovlvl,v_\X^S) \left(\vn_r-\ovlvr-\vn_l+\ovlvl\right)\right] g_\X.
		\end{aligned}
	\end{equation*}
	Note that $ b_2(\vn_l,\vtn)-b_2(\ovlvl,v_\X^S) = q_1 \left(\vn_l-\ovlvl\right) + q_2 \left(\vtn-v_\X^S\right) $ for some smooth and bounded functions $ q_1 $ and $ q_2. $
	Thus, it holds that
	\begin{equation}\label{F-3-4}
	\sum_{k=0}^{1} \int_{-\infty}^{0} \abs{\p_\xi^k(\Fn_3+\Fn_4)}^2 d\xi \leq	\sum_{k=0}^{1} \int_{-\infty}^{0} \left(\abs{\p_\xi^k D_{2,1}^-}^2 + \abs{\p_\xi^k D_{2,2}^-}^2\right) d\xi \leq C\e^2 e^{-2\alpha t}.
	\end{equation}
	Similar to \cref{F-3-4,H-1-plus}, one can prove that
	 \begin{equation*}
	 	\sum_{k=0}^{1} \int_{0}^{+\infty} \abs{\p_\xi^k(\Fn_3+\Fn_4)}^2 d\xi \leq C\e^2 e^{-2\alpha t}.
	 \end{equation*}
	Thus,
		\begin{equation*}
			\begin{aligned}
				\norm{\Hn_2}_1 \leq & ~C \norm{\Hn_1}_2 + \norm{\Fn_3 + \Fn_4}_1
				\leq C\e e^{-\alpha t},
			\end{aligned}
		\end{equation*}
	which finishes the proof of \cref{Lem-H}.


\appendix

\section{Proof of \cref{Lem-periodic}}
Since the local existence is standard, we give only the a priori estimates.
Denote $ \Torus = [0,\pi], $ and $ \norm{\cdot} = \norm{\cdot}_{L^2(\Torus)}, \norm{\cdot}_k = \norm{\cdot}_{H^k(\Torus)}. $
Let $ \phi:= v-\ovlv, \psi:= u-\ovlu, $ then it follows from \cref{N-S} that
\begin{equation}\label{app-equ}
	\begin{cases}
		\p_t \phi = \p_x \psi, \\
		\p_t \psi + \p_x p(v) = \p_x \left( \sigma'(v) \p_x \psi \right).
	\end{cases}
\end{equation}
Define the a priori assumption that $ \norm{\phi,\psi}_k(t) \leq \delta $ for some $ 0<\delta \ll 1, $ where $ k\geq 2. $
Thus, Sobolev inequality yields that 
$$ \sum_{l=0}^{k-1} \norm{\p_x^l\left(\phi,\psi\right)}_{L^\infty(\Torus)} \leq C \delta. $$ 
Multiplying $ \psi $ on \cref{app-equ}$ _2 $ yields that
\begin{equation*}
\begin{aligned}
\p_t \Big(\frac{\psi^2}{2}\Big) + \sigma'(v) \left(\p_x \psi\right)^2 = \p_x\left(\cdots\right) + p(v)\p_x \psi.
\end{aligned}
\end{equation*}
Note that
\begin{equation*}
\begin{aligned}
p(v)\p_x \psi = p(v) \p_t \phi & = \p_t \Big( \int_{\ovlv}^{\ovlv+\phi} p(s)ds- p(\ovlv) \phi \Big) + p(\ovlv) \p_t \phi \\
& = \p_t \Big( \int_{\ovlv}^{\ovlv+\phi} p(s)ds- p(\ovlv) \phi \Big) + \p_x \left( p(\ovlv) u\right),
\end{aligned}
\end{equation*}
which implies that
\begin{equation}\label{app-ineq-1}
\frac{d}{dt} \int_0^\pi \Big(\frac{\psi^2}{2} - \int_{\ovlv}^{\ovlv+\phi} p(s)ds + p(\ovlv) \phi \Big) dx + a_1 \norm{\p_x \psi}^2 \leq 0.
\end{equation}
for some constant $ a_1>0. $
From \cref{app-equ} one has that
\begin{equation}\label{app-phi}
\p_t \psi + \p_x p(v) = \p_x \left( \sigma'(v) \p_t v\right) = \p_t \left( \sigma'(v) \p_x \phi \right).
\end{equation}
Then multiplying \cref{app-phi} by $ \sigma'(v) \p_x \phi $ yields that
\begin{equation*}
\begin{aligned}
& \frac{1}{2}\p_t \left( \sigma'(v) \p_x \phi\right)^2 + \sigma'(v) \abs{p'(v)} \left(\p_x\phi\right)^2 = \sigma'(v) \p_t \psi \p_x\phi \\
&\quad = \p_t\left( \sigma'(v) \psi \p_x \phi \right) - \psi \p_t \left(\sigma'(v) \p_x \phi\right) \\
&\quad =  \p_t\left( \sigma'(v) \psi \p_x \phi \right) - \psi \p_x \left(\sigma'(v) \p_t \phi\right)\\
&\quad = \p_t\left( \sigma'(v) \psi \p_x \phi \right) - \psi \p_x \left(\sigma'(v) \p_x \psi\right)\\
&\quad = \p_t\left( \sigma'(v) \psi \p_x \phi \right) - \p_x \left(\psi \sigma'(v) \p_x \psi\right) + \sigma'(v) \left(\p_x \psi\right)^2,
\end{aligned}
\end{equation*}
which implies that
\begin{equation}\label{app-ineq-2}
\begin{aligned}
\frac{d}{dt} \int_0^\pi  \Big[\frac{1}{2}\left(\sigma'(v) \p_x \phi\right)^2 - \sigma'(v) \p_x \phi \psi \Big] dx + a_2 \norm{\p_x\phi}^2 
\leq C_1 \norm{\p_x \psi}^2,
\end{aligned}
\end{equation}
for some constants $ a_2>0 $ and $ C_1>0. $ Differentiating \cref{app-equ}$ _2 $ with respect to $ x $ and multiplying the resulting by $ \p_x \psi $ yield that
\begin{equation*}
\begin{aligned}
\frac{1}{2} \p_t \left(\p_x \psi\right)^2 + \sigma'(v) \left(\p_x^2 \psi\right)^2 = \p_x\left(\cdots\right) + \p_x p(v) \p_x^2 \psi - \p_x^2 \psi \sigma''(v) \p_x\phi \p_x \psi,
\end{aligned}
\end{equation*}
which implies that
\begin{equation}\label{app-ineq-3}
\begin{aligned}
\frac{d}{dt} \int_0^\pi \left(\p_x \psi\right)^2 dx + a_3 \int_0^\pi \left(\p_x^2 \psi\right)^2 dx & \leq C \left(\norm{\p_x \phi}^2 + \norm{\p_x \phi}_{L^\infty(\R)}^2 \norm{\p_x \psi}^2 \right)\\
& \leq C_2 \left(\norm{\p_x \phi}^2 + \delta^2 \norm{\p_x \psi}^2 \right),
\end{aligned}
\end{equation}
for some constants $ a_3>0 $ and $ C_2>0. $ Choose large constant $ M_2>> M_1>0, $ then $ M_2 \cdot $ \cref{app-ineq-1} $ + M_1 \cdot $ \cref{app-ineq-2} $ + $ \cref{app-ineq-3} gives that
\begin{align}
& \frac{d}{dt} \int_{0}^{\pi} \Big[ \frac{ M_2}{2}\psi^2 + M_2 \Big(-\int_{\ovlv}^{\ovlv+\phi} p(s)ds + p(\ovlv) \phi\Big)+ \frac{ M_1}{2}\left(\sigma'(v) \p_x \phi\right)^2 -  M_1 \sigma'(v) \p_x \phi \psi + \left(\p_x \psi\right)^2 \Big] dx \notag \\
& \qquad + \frac{a_1 M_2}{2} \norm{\p_x \psi}^2 + \frac{a_2 M_1}{2} \norm{\p_x \phi}^2 + a_3 \norm{\p_x^2 \psi}^2 \leq 0. \label{app-ineq-4}
\end{align}
Note that
\begin{equation*}
\abs{ M_1 \sigma'(v) \p_x \phi \psi} \leq \frac{M_1}{4} \left(\sigma'(v) \p_x \phi\right)^2 +  M_1 \psi^2,
\end{equation*}
thus it follows from \cref{app-ineq-4} and Poincar\'{e} inequality that for some $ \alpha>0, $ one has that
\begin{equation*}
\norm{\phi,\psi}_1(t) \leq C \norm{\phi_0,\psi_0}_1 e^{-\alpha t} \leq C\e e^{-\alpha t}.
\end{equation*}
The estimate of the higher order derivatives $ \p_x^l(\phi,\psi) $ with $ l=2,\cdots,k, $ is similar and thus omitted.


\bibliographystyle{amsplain}

\vspace{1.5cm}

\end{document}